\documentclass[12pt]{amsart}
\usepackage[utf8]{inputenc}
\usepackage{setspace}
\usepackage[english]{babel}
\usepackage{mathabx}
\usepackage{framed}
\usepackage{soul}
\usepackage[all]{xy}
\usepackage[a4paper,margin=1in]{geometry}
\usepackage{arydshln}
\usepackage[colorinlistoftodos,textsize=small]{todonotes}
\todostyle{syl}{inline,color=cyan}
\usepackage{pb-diagram}
\usepackage[mathscr]{euscript}
\usepackage{multicol}
\usepackage{stmaryrd} 
\usepackage{empheq} 
\usepackage{tikz-cd}
\tikzset{
  symbol/.style={
    draw=none,
    every to/.append style={
      edge node={node [sloped, allow upside down, auto=false]{$#1$}}}
  }
}

\usepackage{amsmath,amsfonts,amssymb,amsthm,mathrsfs,latexsym,mathtools}
\usepackage{commath}
\usepackage[T1]{fontenc}
\usepackage{hyperref}
\usepackage{caption} 
\usepackage{subcaption} 

\usepackage{lipsum}

\DeclareMathAlphabet{\mathbbmsl}{U}{bbm}{m}{sl}

\usetikzlibrary{matrix}

\newcommand{\C}{\mathbb{C}} 
\newcommand{\Z}{\mathbb{Z}}

\newcommand{\Sym}{\textup{Sym}}
\newcommand{\ddef}{\textup{def}}
\newcommand{\Bl}{\textup{Bl}}

\DeclareMathOperator{\Mod}{\mathrm{Mod}}
\newcommand\ChangeRT[1]{\noalign{\hrule height #1}} 

\usepackage{amscd} 

\usepackage{enumitem} 
\usepackage{comment}

\newtheorem{theorem}{Theorem}[section]
\newtheorem{definition}[theorem]{Definition}
\newtheorem{proposition}[theorem]{Proposition}
\newtheorem{corollary}[theorem]{Corollary}
\newtheorem{question}[theorem]{Question}
\newtheorem{example}[theorem]{Example}
\newtheorem{lemma}[theorem]{Lemma}

\newtheorem{remark}[theorem]{Remark}

\tikzset{commutative diagrams/.cd,
mysymbol/.style={start anchor=center,end anchor=center,draw=none}
}

\title{On the discriminant locus of a generic projection}
\date{April 19, 2025}

\author{Si-Yang Liu}
\address{Hausdorff Center for Mathematics, University of Bonn,
Endenicher Allee 62, Villa Maria, 53115 Bonn, Germany}
\email{syliu@math.uni-bonn.de}

\author{Yilong Zhang}
\address{Department of Mathematics, University of Georgia,
110 Carlton St, Athens, GA 30602, U.S.A.}
\email{Yilong.Zhang@uga.edu}


\begin{document}
\maketitle

\begin{abstract}
For a smooth projective variety $X\subseteq \mathbb P^N$ over an algebraically closed field of char $0$, we show that the discriminant locus of a generic projection of $X$ is projectively dual to a general linear section of the dual variety, and deduce a purity statement for the discriminant. Over $\C$, we also show that the fundamental group of the complement of the branch divisor of a generic projection of a normal hypersurface surjects onto a braid group via braid monodromy.

\end{abstract}

\setcounter{tocdepth}{1} 
\tableofcontents

\section{Introduction}
Let $X$ be a smooth projective variety over an algebraically closed field $\Bbbk$ of char $0$. Given an embedding $X\subseteq \mathbb P^N$, a \textit{Lefschetz pencil} on $X$ is a choice of a general one-parameter family of hyperplane sections of $X$, which induces a rational map 
$$X\dashrightarrow \mathbb P^1.$$
The general member is smooth and the number of the singular members is called the \textit{class} of $X$. When the dual variety $X^{\perp}$ is a hypersurface, this equals $\deg X^{\perp}$ and is computed from the topology of the Lefschetz pencil \cite{Katz73,Lam81}. When $\Bbbk=\C$, the fibration over finite discriminant complement $\mathbb P^1\setminus  \Delta$ induces a monodromy representation
$$\pi_1(\mathbb P^1\setminus \Delta,*)\to \textup{Aut}H^*(X\cap H,\Z),$$
where $X\cap H$ is a smooth hyperplane section, which carries rich topological information. This is the classical Picard--Lefschetz theory \cite{Voisin}.



The motivation of this paper is to generalize this investigation to a Lefschetz net, and more broadly, a higher-dimensional linear subsystem of $|\mathcal{O}_{\mathbb P^N}(1)|$. Such a system arises from projectivization $\mathbb PV$ of a general $(k+1)$ dimensional subspace $V\subseteq H^0(\mathbb P^N,\mathcal{O}(1))$. There are two natural families associated with the linear system. The first comes from observing that $\mathbb PV$ is naturally a subspace of the dual space $(\mathbb P^N)^{\vee}$, so the base change of the universal family of hyperplane sections of $X$ over $(\mathbb P^N)^{\vee}$ to $\mathbb PV$  gives rise to the family
\begin{equation}\label{Intro_eqn_univ-family/PV}
    \mathcal{X}_V\to \mathbb PV,
\end{equation}
whose fiber over $H\in \mathbb PV$ is the hyperplane section $X\cap H$.

The second family is a generic projection 
\begin{equation}\label{Intro_eqn_gen-proj}
    \pi: X\dashrightarrow \mathbb PV^{\vee}, \ x\mapsto [l_0(x):l_1(x):\ldots:l_k(x)],
\end{equation}
where $l_0,\ldots,l_k$ are $k+1$ general linear forms forming a basis of $V$. The base locus is $X\cap \mathbb PV^{\perp}$. 
The two families are related by a linear dual on the ambient projective space. When $k=1$, up to blowing up the base locus, the two families coincide and are precisely the Lefschetz fibration. When $k\ge 2$, they are different: The fiber of generic projection is a higher codimension linear section. We want to explore the relationship between the two families, particularly where the fibers become singular. This is captured by the notion of discriminant.







\begin{definition}[Discriminant locus]\normalfont\label{Intro_def_disc}
    Let $f:\mathcal{Y}\to B$ be a proper dominant morphism. Assume that $B$ is smooth. Define the \textit{discriminant locus} $\Delta\subseteq B$ as the closed subspace consisting of points $b$ where the fiber $f^{-1}(b)=\mathcal{Y}_b$ is singular. In the flat setting, this agrees set-theoretically with the usual definition via Fitting ideal. When $f: \mathcal Y\to B$ is a proper non-dominant map, we use the convention that $\Delta=f(\mathcal Y)$. See more discussion in Section \ref{sec_disc}. 
\end{definition}

For the hyperplane sections family \eqref{Intro_eqn_univ-family/PV}, the discriminant locus is precisely the linear section by $\mathbb PV$ of the dual variety $\mathbb PV\cap X^{\perp}$ \cite{GKZ08}, which is reduced as long as $\mathbb PV$ is in general position, and is moreover irreducible if it has positive dimension. When $k=2$, a general plane section $\mathbb P^2\cap X^{\perp}$ is expected to have only ordinary nodes and cusps, then their numbers can be computed explicitly using intersection theory (cf. \cite[(2.2), (2.3)]{DolLib81}, \cite{Roberts}).

In this paper, we want to investigate similar questions for generic projection:



\begin{question}
    How can one characterize the discriminant locus $\Delta$ of the generic projection of $X$? 
\end{question}
In particular, one may ask about its dimension, irreducibility, degree, and singularities. Here we interpret the discriminant of the rational map \eqref{Intro_eqn_gen-proj} as the discriminant of the regular morphism $\tilde{X}\to \mathbb PV^{\vee}$ obtained by blowing up the base locus.

Our approach is related to the classical theory of polar varieties and generic projections, especially the work of Piene \cite{Pie78}, Kleiman \cite{Kle86}, Lê--Teissier \cite{LeTeissier88}, and Flores--Teissier \cite{FT18}. In particular, the relation between critical loci of projections and polar loci already appears in that literature in various forms.
The new point here is to reinterpret this relation through projective duality, identifying the discriminant of a generic projection with the projective dual of a general linear section of $X^\perp$, and to deduce from this a purity statement and further applications.



\subsection{Duality of discriminant locus}
We begin by addressing the question with projective duality. 
All varieties are over an algebraically closed field of char $0$.


\begin{theorem}[Theorem \ref{thm_mainthmequiv}]\label{thm_dual}
Let $X\subseteq \mathbb P^N$ be a smooth projective variety, and $\mathbb PV$ is a general $k$-dimensional linear subspace in the dual space $(\mathbb P^N)^{\vee}$. 
Then under the natural identification $(\mathbb PV^{\vee})^{\vee}\cong \mathbb PV$, one has
$$\Delta= (\mathbb PV\cap X^{\perp})^{\perp}.$$
In particular, the discriminant loci of two families \eqref{Intro_eqn_univ-family/PV} and \eqref{Intro_eqn_gen-proj} are projectively dual to each other.
\end{theorem}



Note that in the extreme case when $k=N$, the statement reduces to the classical projective duality. Thus Theorem \ref{thm_dual} can be thought of as a relative version of projective duality.

When $\pi$ is not dominant, i.e.\ when $\dim(X)<k$, the map $\pi$ is a generic immersion onto its image,
and the discriminant coincides with the projection image $\pi(X)$.
In this case, a version of Theorem \ref{thm_dual} appears in
\cite[Theorem 1.21]{Tev05} and \cite[Proposition 6.1]{Hol88}. We are mainly interested in the dominant case.

\subsection{Purity of discriminant locus} Recall that the Zariski–Nagata purity theorem (see Theorem \ref{thm_Zariski-Nagata}) implies that a finite morphism from a normal variety to a smooth variety is either branched along a hypersurface, or is étale with empty branch locus. We prove an analogue for the discriminant locus of generic projection.

From now on, $X\subseteq \mathbb P^N$ is a smooth projective variety, and let \begin{equation}\label{intro_eqn_genproj}
    \pi_k:X\dashrightarrow \mathbb P^k
\end{equation} 
be a generic projection to $\mathbb P^k$ and assume it is dominant (equivalently, $\dim(X)\ge k$). For simplicity, we identify $\mathbb PV^{\vee}\cong \mathbb P^k$ once the generic projection is fixed.  


\begin{theorem}[Theorem \ref{Thm2proof}]\label{thm_Dirreducible-hypersurface}
Let $c$ be the codimension of the dual variety $X^{\perp}$ in $(\mathbb P^N)^{\vee}$, then 
    \begin{itemize}
    \item[(i)] if $c<k$, $\Delta(\pi_k)$ is an irreducible hypersurface;
    \item[(ii)] if $c=k$, $\Delta(\pi_k)$ is a finite union of hyperplanes;
    \item[(iii)] if $c>k$, $\Delta(\pi_k)$ is empty.
    \end{itemize}
    In particular, the discriminant locus of a dominant generic projection is either of pure codimension one or is empty.
\end{theorem}

Note that case (i) is generic when $k\ge 2$. Indeed, smooth projective varieties with dual variety not a hypersurface are called \textit{dual defective}, which are special (see Remark \ref{remark_defect}); Case (ii) generalizes Lefschetz pencil in $k=1$; Case (iii) has the following application:





\begin{corollary}[Corollary \ref{cor_smooth-fibration}]\label{intro_cor_smoothfib}
 A smooth projective variety $X\subseteq \mathbb P^N$ is $k$-dual defective (i.e., $\textup{codim}(X^{\perp})>k$) if and only if the extension of the generic projection \eqref{intro_eqn_genproj} to the blowup of base locus $B$
$$\tilde{\pi}:\textup{Bl}_BX\to \mathbb P^k$$
is a smooth morphism. 
\end{corollary}

As a special case, when $k=1$, it implies dual defective varieties admit a smooth Lefschetz pencil. Corollary \ref{intro_cor_smoothfib} recovers a classical result that dual varieties for smooth curves and surfaces are hypersurfaces unless they are linear (see Proposition \ref{prop_no-defect-dim1,2}).

\begin{remark}[Singular total space]\normalfont
 We also show that both Theorem \ref{thm_dual} and \ref{thm_Dirreducible-hypersurface} remain valid when $X$ is normal and $\pi$ is finite and dominant; see Theorem \ref{thm_finite-normal}. When $X$ is singular in general, the claims about duality and purity still apply to a closed subspace of the discriminant captured by the smooth locus of $X$ (see Definition \ref{def_discriminant}, Theorem \ref{thm_mainthmequiv} and \ref{Thm2proof}). 
\end{remark}

\subsection{Monodromy and braid groups.\label{sec:braid-group-monodromy}}
For varieties over $\C$, fundamental groups of complements of hypersurfaces are closely related to the geometry of singular hypersurfaces and have been studied for a long time \cite{Zariskipair,Kampen33,MT88,DolLib81,Moishezon81,Dimca92}. In some of the interesting examples, it's computed that the fundamental groups of hypersurface complements are isomorphic to \emph{braid groups}, $B_m$, which are roughly speaking the group of trajectories of a finite collection of points moving on the sphere $\mathbb S^2\cong \mathbb{P}^1$ (see Section \ref{sec_braidgroup}). 

We extend this perspective to generic finite projections of normal hypersurfaces and relate the resulting branch divisors to projective dual varieties of hyperplane sections.


\begin{theorem}\label{intro_thm_pi1}
Let $Y\subseteq \mathbb P^{N+1}$ be a normal hypersurface of degree $m$, and let
$\pi:Y\to \mathbb P^N$ be a general linear projection with branch divisor $\Delta$.
Then:
\begin{itemize}
    \item $\Delta$ is a hypersurface of degree $\deg(\Delta)=m(m-1)$. If $N>1$, then $\Delta$ is irreducible;
    
    \item let $\mathcal X=Y^\perp\subseteq (\mathbb P^{N+1})^\vee$, 
    and $X=H\cap \mathcal X$ is a general hyperplane section, then $\Delta \cong X^\perp$;

    \item the braid monodromy induces a surjection
  \begin{equation}\label{intro_eqn_surj-pi1}
        \pi_1(\mathbb P^N\setminus \Delta,*)\twoheadrightarrow B_m.
    \end{equation} 
\end{itemize}
\end{theorem}

 
\begin{remark}\normalfont
    Moishezon \cite{Moishezon81} proved that when $Y$ is smooth and more generally if $Sing(Y)$ has codimension at least 3, then the braid monodromy representation \eqref{intro_eqn_surj-pi1} is actually an isomorphism. This fails in general for normal $Y$; see Example \ref{example_B4}.
\end{remark}
\begin{corollary}\label{intro_cor_normaldual}
   Let $\mathcal X$ be a smooth projective variety whose dual variety is a normal hypersurface of degree $m$. Let $X=H\cap \mathcal X$ be a general hyperplane section. Use the identification $\mathbb P^N\cong H^{\vee}$, then there is a surjection 
    $$\pi_1(\mathbb P^N\setminus X^{\perp} ,*)\twoheadrightarrow B_m.$$
\end{corollary}

Smooth varieties with normal dual were studied by Zak \cite{Zak_normaldual}. The normality condition on the dual forces $\mathcal X$ to be non-extendable unless $\mathcal X$ is a quadric hypersurface. Examples of such $\mathcal X$ include determinantal varieties, such as the second Veronese embedding $v_2(\mathbb P^n)$, the Severi varieties \cite{Zak_tangents&secants}, and the Segre embeddings of $\mathbb P^n\times \mathbb P^n$ of bidegree $(1,1)$.\\

\noindent\textbf{Structure of the paper.}
In Section \ref{sec_prelim}, we discuss linear projection of projective space, review the classical projective duality theorem and Pl\"ucker formulas. 

In Section \ref{sec_disc}, we introduce the notion of \textit{smooth discriminant} and compare it with the discriminant of Definition \ref{Intro_def_disc}. We also recall the Zariski--Nagata purity theorem for finite branched covers.

In Section \ref{sec_main-thm} and \ref{sec_purity}, we prove Theorem \ref{thm_dual} and \ref{thm_Dirreducible-hypersurface}. In fact, we prove a more general statement about duality and purity (cf. Theorem \ref{thm_mainthmequiv} and \ref{Thm2proof}) for the smooth discriminant of generic projection of an irreducible projective variety $X$. 

Section \ref{sec_emptydisc}, \ref{sec_example} and \ref{sec_polar} are applications of the main theorems.
In Section \ref{sec_emptydisc}, we prove Corollary \ref{intro_cor_smoothfib} and discuss failure in some singular setting. In Section \ref{sec_example}, we demonstrate Theorem \ref{thm_dual} in surfaces and threefolds, and compute singularities of discriminant curves. In Section \ref{sec_polar}, we discuss the relationship to the polar locus and deduce some classical results on the degree of polar loci. 



In Section \ref{sec_braidgroup}, we study braid monodromy representations of fundamental group of complement of irreducible hypersurface and prove Theorem \ref{intro_thm_pi1}.
\\
 
 \noindent\textbf{Field of definition.} Throughout the paper, we work over an algebraically closed field $\Bbbk$ of characteristic $0$, and particularly over $\C$ in Section \ref{sec_braidgroup}. Note that the arguments in Section \ref{sec_main-thm} (in particular Theorem \ref{thm_dual} and \ref{thm_mainthmequiv}) rely only on biduality for projective duals, and should extend to positive characteristic for reflexive varieties. \\
 
\noindent\textbf{Acknowledgment.}  We would like to thank Anatoly Libgober, Christian Schnell, John Sheridan, and Bernard Teissier for helpful conversations and correspondence. We are grateful to Thomas Dedieu for pointing out the distinction between our definition of smooth discriminant (Definition \ref{def_discriminant}) and the definition via Fitting ideals in the literature, for carefully reading an earlier draft, and for many helpful comments on the revision. 

We gratefully acknowledge the support and hospitality of the Simons Center for Geometry and Physics at Stony Brook University during the \textit{2nd Simons Math Summer Workshop: Moduli} in July 2024, where a project motivating this work began. S.-Y. Liu is funded by the Deutsche Forschungsgemeinschaft (DFG, German Research Foundation) under Germany's Excellence Strategy - GZ 2047/1, Projekt-ID 390685813.



\section{Preliminaries}\label{sec_prelim}
In this section, we revisit classical projective duality from the viewpoint of linear algebra and conormal varieties, setting up the main geometric framework for our later results. 


\subsection{Linear algebra}
For vector space $\Bbbk^{N+1}$, we denote $(\Bbbk^{N+1})^{\vee}$ its dual space consisting of linear functionals on $\Bbbk^{N+1}$. For any linear subspace $V\subseteq \Bbbk^{N+1}$, we denote $V^{\perp}\subseteq (\Bbbk^{N+1})^{\vee}$ be the subspace consisting of linear functionals vanishing on $V$. Note that $\textup{codim}(V^{\perp})=\dim(V)$.

\begin{lemma}\label{lemma_V*and-quotient}
 Let $V^{\vee}$ denote the dual space of $V$. Then there is a natural isomorphism of vector spaces
  $$V^{\vee}\cong (\Bbbk^{N+1})^{\vee}/V^{\perp}.$$
\end{lemma}
\begin{proof}
    Consider the exact sequence 
$$0\to V\to \Bbbk^{N+1}\to \Bbbk^{N+1}/V\to 0.$$
Dualizing yields
    $$0\leftarrow   V^{\vee}\leftarrow  (\Bbbk^{N+1})^{\vee}\leftarrow   (\Bbbk^{N+1}/V)^{\vee}\leftarrow0.$$
The subspace $(\Bbbk^{N+1}/V)^{\vee}$ consists of those linear functionals on $\Bbbk^{N+1}$ that vanish on $V$. Hence, it is canonically identified with $V^{\perp}$. This proves the claim.
\end{proof}

Passing to projectivization, $\mathbb PV\subseteq \mathbb P^N$ is a projective subspace, then $\mathbb PV^{\perp}\subseteq (\mathbb P^N)^{\vee}$ and induces the rational projection
\begin{equation}\label{eqn_p}
    p: (\mathbb P^N)^{\vee}\dashrightarrow \mathbb P \big 
   ((\Bbbk^{N+1})^{\vee}/V^{\perp}\big )\cong \mathbb PV^{\vee}.
\end{equation}
The indeterminacy locus of $p$ is exactly $\mathbb PV^{\perp}$.

Identifying linear functionals with the set of hyperplanes in $\mathbb P^N$, the map \eqref{eqn_p} can be described by sending a hyperplane $H$ of $\mathbb P^N$ to a hyperplane in $\mathbb PV$ by intersection:
$$H\mapsto H\cap \mathbb PV.$$
The image defines a hyperplane in $\mathbb PV$ if and only if $H$ is transversal to $\mathbb PV$, if and only if $\mathbb PV\nsubseteq H$. 
\begin{lemma}\label{lmm_linear_PD}
    Let $x\in\mathbb{P} V\subseteq\mathbb{P}^N$ be a point with dual hyperplane $x^{\perp}\subseteq\left(\mathbb{P}^N\right)^{\vee}$, then the projection $p\colon\left(\mathbb{P}^N\right)^{\vee}\dashrightarrow\mathbb{P}V^{\vee}$ send the hyperplane $x^{\perp}$ in $(\mathbb P^N)^{\vee}$ to the dual hyperplane $x_V^{\perp}$ of $x$ in $\mathbb{P} V$. Conversely, $x^{\perp}$ is the cone of $x_V^{\perp}$ with center being the base locus $\mathbb PV^{\perp}$.
\end{lemma}
\begin{proof}
    Note that the dual of any $\alpha\in x^{\perp}\setminus\mathbb{P} V^{\perp}$ is a hyperplane $\alpha^{\perp}\subseteq\mathbb{P}^N$ containing $x$ and is transversal to $\mathbb{P}V$, therefore the intersection $\alpha^{\perp}\cap\mathbb{P}V$ is a hyperplane of $\mathbb{P}V$ whose dual is $p(\alpha)\in\mathbb{P}V^{\vee}$. Let $x_V^{\perp}$ be the dual of $x$ in $\mathbb{P}V$. As $x\in\alpha^{\perp}\cap\mathbb{P}V$, we know that $p(\alpha)\in x_V^{\perp}$, and therefore $p(x^{\perp})\subseteq x_V^{\perp}$. Dimension counting tells us that $p(x^{\perp})=x_V^{\perp}$.
\end{proof}

\subsection{Projective duality}

Now let $X\subseteq \mathbb P^{N}$ be an irreducible subvariety. Let $x\in X^{sm}$ be a closed point of $X$ contained in the smooth locus $X^{sm}$, then the tangent space $T_xX$ at $x$ has the same dimension as $X$. We may regard $T_xX$ as a projective subspace of $\mathbb P^{N}$ of dimension $\dim(X)$. We use this convention throughout the paper. Using our notation, $(T_xX)^{\perp}\subseteq (\mathbb P^N)^{\vee}$ is the conormal space and has dimension $N-\dim(X)-1$.

The \textit{conormal variety} $N_{\mathbb P^N}^{\perp}X$ of $X$ is defined to be the closure of the union of conormal spaces at smooth points \cite{HK85}
\begin{equation}\label{eqn_conormal}
    N_{\mathbb P^N}^{\perp}X=\overline{\{(x,H)\in X^{sm}\times (\mathbb P^{N})^{\vee}| T_xX\subseteq H\}}.
\end{equation}
It records pairs of smooth points and their tangent hyperplanes.
It is irreducible and has dimension $N-1$. Note that here, to require the hyperplane $H$ to contain $T_xX$ is equivalent to finding a hyperplane $H$ tangent to $X$ at $x$. In the case when $X$ is a hypersurface $f=0$ of degree at least two, and $x$ is a smooth point, such $H$ is unique and coincides with the tangent hyperplane $T_xX$, which is determined by the normal vector $\nabla f_x$.

\begin{definition}\normalfont
   The \textit{dual variety} of $X$, denoted $X^{\perp}$ throughout the paper, is defined to be the image of the projection map $N_{\mathbb P^N}^{\perp}X\to (\mathbb P^{N})^{\vee}$. 
\end{definition}

\begin{remark}\normalfont It is more common in the literature to denote the dual variety by $X^*$ or $X^{\vee}$. We instead write $X^{\perp}$ to emphasize that, by definition, $X^{\perp}$ is the (closure of) the union of the conormal spaces
\[
 (T_x X)^{\perp} \subset (\mathbb P^N)^{\vee}, \qquad x \in X^{\mathrm{sm}}.
\]
In particular, we reserve the notation $(\ \cdot\ )^{\vee}$ for dual vector spaces, and use $(\ \cdot\ )^{\perp}$ for orthogonal complements and for dual varieties.
\end{remark}

There is an incidence correspondence between $X$ and $X^{\perp}$ shown below.

\begin{figure}[ht]
    \centering
\begin{tikzcd}
&N_{\mathbb P^N}^{\perp}X\arrow[dl,"\pi_1"']\arrow[dr,"\pi_2"]\\
X && X^{\perp}\subseteq (\mathbb P^N)^{\vee}.
\end{tikzcd}
\end{figure}
Hence, the conormal space $(T_{x_0}X)^{\perp}$ is precisely the fiber $\pi_1^{-1}(x_0)$, and is identified with $\pi_2(\pi_1^{-1}(x_0))$, a linear subvariety of the dual variety $X^{\perp}$. 

The classical projective duality is a reflexivity property:
\begin{theorem}[Classical projective duality \cite{Kle86}] \label{thm_classicalPD}
Let $X$ be a projective variety over an algebraically closed field of char $0$. Then the dual variety of $X^{\perp}$ is isomorphic to $X$, i.e., there is an isomorphism
    $$X^{\perp\perp}\cong X.$$
\end{theorem}

Equivalently, projective duality says that on a Zariski dense open subspace of conormal variety 
$$(N_{\mathbb P^N}^{\perp}X)^{\circ}=N_{\mathbb P^N}^{\perp}X\cap (X^{sm}\times (X^{\perp})^{sm}),$$
a pair  $(x,\alpha)\in (N_{\mathbb P^N}^{\perp}X)^{\circ}$ satisfies 
\begin{equation}\label{eqn_classical-proj-dual}
    T_xX\subseteq \alpha^{\perp}\ \Longleftrightarrow\ T_\alpha(X^{\perp})\subseteq x^{\perp},
\end{equation}
where $\alpha^{\perp}$ (resp. $x^{\perp}$) is the corresponding hyperplane in $\mathbb P^{N}$ (resp. $(\mathbb P^{N})^{\vee}$).

The proof is to show there is an isomorphism of conormal varieties $N_{\mathbb P^N}^{\perp}X\cong N_{(\mathbb P^N)^{\vee}}^{\perp}X^{\perp}$ by switching the two coordinates. Over $\C$, the affine cone of the conormal variety is a conical Lagrangian subvariety, and the biduality is to switch the two factors. One refers to \cite[Section 1.3]{Tev05}, \cite{GKZ08} for details of this perspective.

\begin{example}\normalfont
    Let $X$ be a smooth hypersurface of $\mathbb P^{N}$ of degree $d\ge 2$, then its dual variety is also a hypersurface (in $(\mathbb P^N)^{\vee}$) with degree \cite[Proposition 2.9]{EisHar16}
    \begin{equation}\label{eqn_deg(X-dual)}
        \deg(X^{\perp})=d(d-1)^{N-1}.
    \end{equation}
       
  When $d\ge 3$, the dual hypersurface $X^{\perp}$ is singular in codimension one, and in particular non-normal. As an example, when $X$ is a smooth cubic curve in $\mathbb P^2$, its dual curve is a sextic curve with 9 cusps. 
  
    Projective duality implies that if a pair $(x,H)\in X^{sm}\times (X^{\perp})^{sm}$ is contained in the smooth locus of the product, then $H$ is tangent to $X$ at $x$ iff $x^{\perp}$ as a hyperplane in $(\mathbb P^N)^{\vee}$ is tangent to $X^{\perp}$ at $H$.
 
\end{example}

\begin{remark}[Dual defect]\normalfont\label{remark_defect}
    For "most" smooth projective varieties $X$, the dual variety is a hypersurface. Indeed, this happens iff there exists a hyperplane $H$ such that $H\cap X$ has an ordinary node. 
    
    In general, for a smooth projective variety $X$, the \textit{dual defect} is defined to be $$\ddef(X):=N-1-\dim(X^{\perp}),$$
   which measures how far the dual variety is from a hypersurface. Varieties with positive defects are often called \textit{dual defective varieties}.

   For example, in dimensions one and two, dual defective varieties are linear. The classification is done up to dimension 10 (cf. \cite{Ein86,Ein85,LanStr87,BMF92}). We will also study defect from the point of view of generic projection in Section \ref{sec_smoothfib}.


\end{remark}
\begin{remark}[Char $p$]\normalfont
   In positive characteristic, Theorem \ref{thm_classicalPD} is not valid in general. Those varieties that satisfy Theorem \ref{thm_classicalPD} are called \textit{reflexive varieties}. Many theorems on dual varieties in char $0$ still hold for reflexive varieties in char $p$. One refers to \cite{Kle86,HK85} for details.
\end{remark}

\subsection{Pl\"ucker formula}
Let $C\subseteq \mathbb P^2$ be a planar curve of degree $d$ with only ordinary nodes and cusps and let their numbers be $\delta$ and $\kappa$ respectively. Then the degree $d^{\perp}$ and genus $g^{\perp}$ of the dual curve $C^{\perp}$ satisfy Pl\"ucker relations (cf.  \cite{Dolgachev-classicalAG} \cite{Kle77})
 \begin{align}
d^{\perp}&=d(d-1)-2\delta-3\kappa,\label{eqn_Plucker1}\\
g^{\perp}&=\frac{(d-1)(d-2)}{2} -\delta-\kappa.\label{eqn_Plucker2}
\end{align}
In addition, if the dual curve $C^{\perp}$ has only ordinary nodes and cusps with numbers $\delta^{\perp}$ and $\kappa^{\perp}$, then they correspond to the number of  bitangents and flex lines on $C$, respectively. Moreover, the degree $d$ and genus $g$ of $C$ satisfy the dual version of \eqref{eqn_Plucker1} and \eqref{eqn_Plucker2} by switching the invariants of $C$ and $C^{\perp}$.


\section{Discriminant of an algebraic map}\label{sec_disc}
For a proper flat family $\pi:\mathcal Y\to B$, the discriminant of Definition \ref{Intro_def_disc} agree set-theoretically with the discriminant defined via the $d$-th Fitting ideal of $\Omega_{\mathcal Y/B}$, where $d=\dim(\mathcal Y)-\dim(B)$ is the relative dimension. See \cite{Tei77} \cite[\href{https://stacks.math.columbia.edu/tag/0C3K}{0C3K}]{stacks-project}. In this section, we use $\Delta_{Fitt}(\pi)$ to denote this discriminant. In this paper, we only focus on the set-theoretical properties of this discriminant and treat $\Delta_{Fitt}(\pi)$ as a variety.
\subsection{Smooth discriminant}
Here we introduce another definition of discriminant of $\pi$ below, which in general is a component of $\Delta_{Fitt}(\pi)$.
\begin{definition}[Smooth critical and discriminant locus]\normalfont \label{def_discriminant}
Let $\pi:\mathcal Y\to B$ be a proper morphism of varieties with $B$ smooth.
Define the \textit{smooth critical locus} $\mathcal C(\pi)$ to be the closure of the set of points $y\in \mathcal Y^{sm}$
such that $d\pi_y:T_y\mathcal Y\to T_{\pi(y)}B$ is not surjective, where $\mathcal Y^{sm}$ is the smooth part of $\mathcal Y$.
Define the \textit{smooth discriminant locus} to be the image  $$\Delta(\pi):=\pi(\mathcal C(\pi))\subseteq B.$$
Equivalently, $\Delta(\pi)$ is the closure of critical values of $\mathcal Y^{sm}\to B$.

\end{definition}
Note that our definition of $\Delta(\pi)$ is set-theoretical, i.e., is a closed variety of $B$. For the generic projection $\pi:X\dashrightarrow \mathbb P^k$, we can extend it to a morphism by blowing up the base locus $\tilde{\pi}:\tilde{X}=\textup{Bl}_BX\to \mathbb P^k$. The smooth discriminant of $\pi$ is referred to as the smooth discriminant of $\tilde{\pi}$.
\begin{remark}\normalfont
\begin{itemize}
    \item[(1)] 
    When the total space $\mathcal Y$ is smooth (e.g., blow-up of a generic projection of a smooth $X$), then  $\mathcal{C}(\pi)$ is the underlying set of the critical scheme and parameterizes the points at which tangent map
$$(d\pi)_y\colon T_y\mathcal{Y}\to T_{\pi(y)} B$$
is not surjective. If $\pi$ is moreover flat, then $\Delta(\pi)=\Delta_{Fitt}(\pi)$ is the classical (Fitting) discriminant (cf. Definition \ref{Intro_def_disc}). 

\item[(2)] When $\pi$ is not dominant, $\Delta(\pi)=\pi(\mathcal Y)$ is the image of $\mathcal Y$. This is different from the conventions in the literature, e.g., \cite[Definition 2.15]{Shi06}. 
\item[(3)] A similar notion of smooth critical locus $\mathcal C(\pi)$ has also been studied by Flores and Teissier in the affine complex analytic setting \cite{FT18}.
\end{itemize}
\end{remark}

 We note the following relations between the smooth discriminant $\Delta$ and the (Fitting) discriminant $\Delta_{Fitt}$.







\begin{lemma}\label{lemmma_disc_vs_smoothdisc}
Let $\pi: \mathcal Y\to B$ be a proper flat morphism and assume that $B$ is smooth. Then set-theoretically, the discriminant $\Delta_{Fitt}(\pi)$ is the union of the smooth discriminant and the image of the singularities of the total space:
    $$\Delta_{Fitt}(\pi)=\Delta(\pi)\cup \pi(Sing(\mathcal Y)).$$
In particular, when the total space $\mathcal {Y}$ is smooth, then two discriminants coincide: $$\Delta_{Fitt}(\pi)=\Delta(\pi).$$ 
\end{lemma}
\begin{proof}
    When $y\in \mathcal Y^{sm}$,  since $\pi$ is flat and $B$ is smooth, $\pi$ is smooth at $y$ iff $d\pi_y$ is surjective. If $y\in Sing(\mathcal Y)$, then $\pi$ cannot be smooth at $y$. Hence we have $$V(Fitt_d(\Omega_{\mathcal Y/B}))=\mathcal C(\pi)\cup Sing(\mathcal Y).$$
    Then apply projection $\pi$, we obtain the claim.
\end{proof}
In general, we note that the smooth discriminant $\Delta(\pi)$ may not fully capture the information of singularities of $\mathcal Y$ as in the discriminant $\Delta_{Fitt}(\pi)$. See discussions in Remark \ref{remark_twistedcubicdual}.

\subsection{Finite branched cover}
In the case when $\pi$ is a branch cover ($d=0$), we recall the Zariski--Nagata purity theorem:
\begin{theorem}[Purity of the branch cover \cite{Zariski-purity,Nagata-purity}]\label{thm_Zariski-Nagata}
  Let $\pi: \mathcal{Y} \to B$ be a finite surjective morphism where $B$ is a smooth variety and $\mathcal{Y}$ is a normal variety. Let $U \subseteq B$ be the maximal open subset such that the restricted morphism $\pi^{-1}(U) \to U$ is étale. Then the branch locus $B\setminus U$ agrees set-theoretically with $\Delta_{Fitt}(\pi)$. Moreover, $\Delta_{Fitt}(\pi)$ is either empty or is a divisor of pure codimension one in $B$.
\end{theorem}

\begin{definition}\normalfont
    For a dominant proper morphism $\pi:\mathcal Y\to B$, we call its discriminant \textit{clean} if $\pi(Sing(\mathcal Y))$ is contained in the smooth discriminant, in other words, if set-theoretically one has
    $$\Delta_{Fitt}(\pi)=\Delta(\pi).$$
In this case, we will unambiguously call $\Delta(\pi)$ the discriminant of $\pi$.
\end{definition}
In Lemma \ref{lemmma_disc_vs_smoothdisc}, we see that the discriminant is clean when $\mathcal Y$ is smooth. We provide another case.
\begin{lemma}\label{lemma_finite-clean}
   Let $X\subseteq \mathbb P^N$ be a normal projective variety of dimension $n$. Let $$\pi:X\to \mathbb P^n$$ be the generic projection (note this is a finite dominant morphism). Then the discriminant $\Delta(\pi)$ is clean.
\end{lemma}
\begin{proof}
   By our assumption that $X$ is normal and projection center is general, then the singular locus of the total space $Sing(X)$ is finite onto its projection image, and hence has image at most $n-2$. This is a closed subspace of $\Delta_{Fitt}(\pi)$, but by the purity Theorem \ref{thm_Zariski-Nagata}, this must be contained in $\Delta(\pi)$.
\end{proof}

\begin{remark}\normalfont
   For all the results in the Introduction, the discriminant of a generic projection is clean, hence $\Delta(\pi)=\Delta_{Fitt}(\pi)$ is unambiguous. 
\end{remark}


For the rest of the paper, we use $\Delta(\pi)$ to denote the smooth discriminant of the generic projection $\pi:X\dashrightarrow\mathbb P^k$ of an irreducible projective variety $X$.

\section{Duality of smooth discriminant}\label{sec_main-thm}
In this section, we prove Theorems \ref{thm_dual}. In fact, we will prove a general statement (Theorem \ref{thm_mainthmequiv}) for projective duality of the \textit{smooth discriminant} of a generic projection of an \textit{irreducible} projective variety $X\subseteq\mathbb P^N$ over an algebraically closed field with char $0$. $X$ is not necessarily smooth.

\subsection{The setup}
 Throughout this section, we use the convention opposite to that of the introduction: We take a linear section of $X$ and a generic projection of the dual variety $X^{\perp}$ to a lower-dimensional linear subspace. (Note that by classical projective duality, we can switch the roles of $X$ and $X^{\perp}$.)

Let $\mathbb PV$ be a linear subspace of $\mathbb P^N$.
We consider the intersection $X\cap \mathbb PV$ and denote it as $X_V$. We denote $X_{V}^{\perp}$ the dual variety with respect to the embedding $X_V\subseteq \mathbb PV$. We denote $I_\mathbb P^N$ to be the "0-th order" incidence subvariety
$$I_{\mathbb P^N}:=\{(x,H)\in \mathbb P^N\times (\mathbb P^N)^{\vee}: x\in H\}.$$
Similar for $I_{\mathbb PV}$. The conormal varieties $N_{\mathbb P^N}^{\perp}X$ and $N^{\perp}_{\mathbb PV}(X_V)$ are the closure of unions of conormal spaces at smooth points of $X$ and $X_V$ respectively \eqref{eqn_conormal}. They can be regarded as "1st order" incidence variety and are naturally included in $I_{\mathbb P^N}$ and $I_{\mathbb PV}$. 

\begin{figure}[ht]
    \centering
\begin{tikzcd}
&&N_{\mathbb P^N}^{\perp}X\arrow[dll,"\pi_1"']\arrow[drr,"\pi_2"]\\
X \arrow[r,symbol=\subseteq] &\mathbb P^{N} &I_{\mathbb P^N}\arrow[l]\arrow[r]\arrow[u,symbol=\supseteq]&
(\mathbb P^N)^{\vee}\arrow[d,dashrightarrow,"p"]\arrow[r,symbol=\supseteq]& X^{\perp} \arrow[dl,dashrightarrow,"\pi"]\\
X_V\arrow[r,symbol=\subseteq]\arrow[u,symbol=\subseteq]&\mathbb PV\arrow[u,symbol=\subseteq] &I_{\mathbb PV}\arrow[l]\arrow[r]\arrow[d,symbol=\supseteq]& (\mathbb PV)^{\vee}\arrow[r,symbol=\supseteq] & X^{\perp}_V\\
&&N^{\perp}_{\mathbb PV}(X_V)\arrow[ull]\arrow[urr]
\end{tikzcd}
\caption{Linear section of $X$ and generic projection of $X^{\perp}$\label{fig:generic-projection}}
\end{figure}
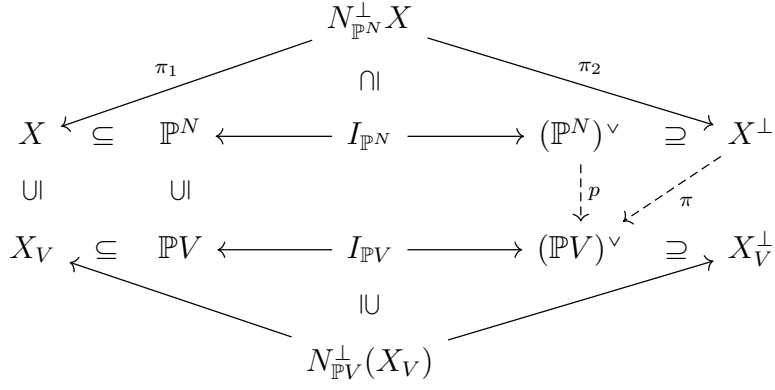

Natural inclusions, projections and correspondence form the commutative diagram \ref{fig:generic-projection}. Here, $p$ is the projection from the center $\mathbb PV^{\perp}$. The target is identified to $(\mathbb PV)^{\vee}$ according to Lemma \ref{lemma_V*and-quotient}. The rational map $\pi$ is the restriction of $p$ to the dual variety $X^{\perp}$.

In the setup described above, our main result is 

\begin{theorem}[cf. Theorem \ref{thm_dual}]\label{thm_mainthmequiv} Let $X\subseteq\mathbb P^N$ be an irreducible projective variety.
For general choice of linear subspace $\mathbb PV\subseteq \mathbb P^N$, the smooth discriminant $\Delta(\pi)$ of generic projection $$\pi:X^{\perp}\dashrightarrow  (\mathbb PV)^{\vee}$$ coincides with the dual variety $X^{\perp}_V$ of the linear section $X_V=X\cap \mathbb PV$ of $X$. In other words, 
$$\Delta(\pi)= X_V^{\perp}$$
as subvarieties of $\mathbb PV^{\vee}$.
\end{theorem}

Note that Theorem \ref{thm_dual} follows from the above, together with classical projective duality \ref{thm_classicalPD}. To prove Theorem \ref{thm_mainthmequiv}, the key is to describe the smooth critical locus of the generic projection $\pi$ using an incidence relation (cf. Proposition \ref{prop_critical-locus}). Then the proof follows from classical projective duality and the generality lemma for the choice of $\mathbb PV$.

\subsection{The smooth critical locus}
We begin with introducing a subvariety of the dual variety $X^{\perp}$ that captures the information of the dual correspondence of $X$ on the subspace $\mathbb PV$.
\begin{definition}\normalfont\label{def_critlocus}
    We define the subvariety $(X|_V)^{\perp}$ of $X^{\perp}$ to be 
    $$(X|_V)^{\perp}:=\pi_2(\pi_1^{-1}(X_V))=\overline{\bigcup_{x\in X_V^{sm}}(T_xX)^{\perp}}.$$
    In other words, $(X|_V)^{\perp}$ is the closure of the union of conormal spaces of $X$ at points contained in $\mathbb PV$.
\end{definition}

\begin{lemma}[Generality of the linear section]\label{lemma_generalityV}
Let
\[
U:=(N_{\mathbb P^N}^{\perp}X)^\circ
\subset N_{\mathbb P^N}^{\perp}X
\]
be the dense open subset consisting of pairs $(x,\alpha)$ such that
$x\in X^{sm}$, $\alpha\in (X^\perp)^{sm}$, and $T_xX\subseteq \alpha^\perp$.
Then for a general linear subspace $\mathbb PV\subset \mathbb P^N$, the following hold:

\begin{itemize}
\item[(1)] The intersection
$U\cap \pi_1^{-1}(X_V)$
is dense in $\pi_1^{-1}(X_V)$.

\item[(2)] The intersection
$X^{sm}\cap X_V$
is dense in $X_V$.
\end{itemize}

Consequently:

\begin{itemize}
\item[(i)] a general point $\alpha\in (X|_V)^\perp$ satisfies
\[
\alpha\in (X^\perp)^{sm},
\qquad
\exists\, x\in X^{sm}\cap \mathbb PV
\text{ such that } T_xX\subseteq \alpha^\perp;
\]

\item[(ii)] a general point $\alpha_V\in X_V^\perp$ is represented by a hyperplane
$H_V\subset \mathbb PV$ tangent to $X_V$ at some smooth point 
$x\in X_V^{sm}\cap X^{sm}$.
\end{itemize}
\end{lemma}

\begin{proof}
Let $r=N-k$. Choose general hyperplanes $H_1,\dots,H_r\subset \mathbb P^N$
cutting out $\mathbb PV$, and write
$D_i:=\pi_1^*H_i \in |\pi_1^*\mathcal O_{\mathbb P^N}(1)|$. Then
\[
\pi_1^{-1}(X_V)=N_{\mathbb P^N}^{\perp}X\cap D_1\cap\cdots\cap D_r.
\]
Since $U$ is a dense open subset of the irreducible variety
$N_{\mathbb P^N}^{\perp}X$, we may choose the divisors $D_i$ inductively so that
no irreducible component of the intermediate intersections is contained in the
closed subset $N_{\mathbb P^N}^{\perp}X\setminus U$.
It follows that
$U\cap \pi_1^{-1}(X_V)$
is dense in $\pi_1^{-1}(X_V)$, proving (1).
Applying $\pi_2$, whose image is $(X|_V)^\perp$, gives the consequence (i).

Similarly, viewing $X_V=X\cap H_1\cap\cdots\cap H_r$, we may choose the same
general hyperplanes inductively so that no irreducible component of the
intermediate intersections is contained in the closed subset $Sing(X)\subset X$.
Therefore $X^{sm}\cap X_V$
is dense in $X_V$, proving (2).
For (ii), note that
$X_V^{sm}\cap X^{sm}$
is a dense open subset of $X_V$ by (2), hence also dense in $X_V^{sm}$. The rest follows from definition of the dual variety $X_{V}^{\perp}$.
\end{proof}

We first observe that
\begin{lemma}\label{lemma_maptodual}
    The $\pi$-image of $(X|_V)^{\perp}$ is contained in $X_{V}^{\perp}$. In other words, there is the following commutative diagram
\begin{figure}[ht]
    \centering
\begin{equation}\label{diagram_crit-locus}
\begin{tikzcd}
X^{\perp}\arrow[d,dashrightarrow,"\pi"]\arrow[r,hookleftarrow]&(X|_V)^{\perp} \arrow[d,dashrightarrow,"\pi|_{(X|_V)^{\perp}}"]\\
(\mathbb PV)^{\vee} \arrow[r,hookleftarrow]& X^{\perp}_V,
\end{tikzcd}
\end{equation}
\end{figure}
   where the horizontal maps are inclusions, and the vertical ones are rational maps restricted from the linear projection $p$.
\end{lemma}
\begin{proof}
Note that for $x\in X_V^{sm}$, a hyperplane $H\subseteq \mathbb P^N$ containing the tangent space $T_xX$ also contains $T_x(X_V)$. Since $T_x(X_V)=T_xX\cap \mathbb PV$, we have $H\cap \mathbb PV$ contains $T_x(X_V)$ as well. In addition, by \eqref{eqn_p}, as long as $H\nsupseteq \mathbb PV$ (which is equivalent to $H^{\perp}$ is not in the base locus of $\pi$), $\pi(H^{\perp})$ defines a hyperplane $(H\cap \mathbb PV)^{\perp}$ of $\mathbb PV$. Now by Lemma \ref{lemma_generalityV}, a general point $H^{\perp}
\in (X|_V)^{\perp}$ corresponds to a smooth point $x$ of $X_V$ and satisfies $T_xX\subseteq H$, then $T_x(X_V)\subseteq H\cap \mathbb PV$. In other words, $\pi(H^{\perp})$ defines a hyperplane tangent to $X_V$ at $x$, hence $\pi(H^{\perp})\in X_{V}^{\perp}$.
\end{proof}

To describe the map $\pi|_{(X|_V)^{\perp}}$, we choose a general point $x\in X_V$ such that $x$ is a smooth point for both $X_V$ and $X$, then by the description of the map \eqref{eqn_p}, the restriction $\pi|_{(X|_V)^{\perp}}$ to $(T_xX)^{\perp}$ is a linear map between two conormal spaces at $x$:
\begin{figure}[ht]
    \centering
\begin{equation}\label{eqn_conormal-map}
\begin{tikzcd}
(T_xX)^{\perp}\arrow[equal]{r}\arrow[d,"\pi|_{(T_xX)^{\perp}}"]& \{H\in (\mathbb P^N)^{\vee}|x\in H\ \textup{and}\ T_xX\subseteq H\}\arrow[d]& H\arrow[d,mapsto]\arrow[l,symbol=\in]\\
(T_xX_V)^{\perp} \arrow[equal]{r}& \{H_{V}\in \mathbb PV^{\vee}|x\in H_V\ \textup{and}\ T_x(X_V)\subseteq H_V\}&H\cap \mathbb PV\arrow[l,symbol=\in].
\end{tikzcd}
\end{equation}
\end{figure}



Then up to further restricting to a Zariski open subspace, $\pi|_{(X|_V)^{\perp}}$ is the union of all the linear maps \eqref{eqn_conormal-map} when $x$ runs in smooth locus of $X_V$:
\begin{equation}\label{eqn_union}
    \bigcup_{x\in X_V^{sm}}(T_xX)^{\perp}\to \bigcup_{x\in X_V^{sm}}(T_xX_V)^{\perp}.
\end{equation}

Next we observe that
\begin{proposition}\label{prop_critical-dom}
  The map  $\pi|_{(X|_V)^{\perp}}$  is dominant.
\end{proposition}
\begin{proof}
A general point $\alpha_V\in X_V^\perp$ is represented by a hyperplane
$H_V\subset \mathbb PV$ tangent to $X_V$ at some point
$x\in X_V^{sm}\cap X^{sm}$ by Lemma \ref{lemma_generalityV} (2). In particular, the hyperplane $H$ of $\mathbb{P}^N$ containing 
\begin{equation}\label{eqn_section}
   \textup{Span}(H_V,T_xX)
\end{equation} is tangent to $X$ at $x$ and satisfies $H\cap \mathbb PV=H_V$. Hence, $H^{\perp}\in (X|_V)^{\perp}$ is in the preimage of $\alpha_V$, and the claim follows.
\end{proof}


By blowing up the base locus $X^\perp\cap \mathbb PV^\perp$, the rational map $\pi:X^{\perp}\dashrightarrow\mathbb PV^{\vee}$
extends to a morphism
\[
\tilde\pi:\widetilde{X^\perp}\to (\mathbb PV)^\vee.
\]
Recall that we defined the smooth critical locus $\mathcal{C}(\tilde{\pi})$ for $\tilde{\pi}$ (see Definition \ref{def_discriminant}).

\begin{definition}\normalfont\label{def_smoothcrit_rat}
We define the \textit{smooth critical locus of the rational map $\pi$}, denoted
$\mathcal C(\pi)\subset X^\perp$, to be the closure in $X^\perp$ of the set of
points
$\alpha\in (X^\perp)^{sm}\setminus \mathbb PV^\perp$
such that
\[
d\pi_\alpha:T_\alpha(X^\perp)\to T_{\pi(\alpha)}(\mathbb PV)^\vee
\]
is not surjective.
Equivalently, $\mathcal C(\pi)$ is the image in $X^\perp$ of $\mathcal{C}(\tilde{\pi})$ and they agree away from the exceptional divisor.
\end{definition}

Our key observation is the following.
\begin{proposition}\label{prop_critical-locus}
The smooth critical locus $\mathcal C(\pi)\subset X^\perp$ of the rational map
\[
\pi:X^\perp\dashrightarrow (\mathbb PV)^\vee
\]
coincides with $(X|_V)^\perp$.
\end{proposition}

We first need to show the following lemma.

\begin{lemma}\label{lemma_C(pi)general}
Let $\alpha\in (X^\perp)^{sm}\setminus \mathbb PV^\perp$, and let
\[
L_\alpha:=p^{-1}(\pi(\alpha))
=\operatorname{Span}(\alpha,\mathbb PV^\perp).
\]
Then the following are equivalent:
\begin{itemize}
\item[(1)] $\alpha\in \mathcal C(\pi)$;
\item[(2)] the differential $d\pi_\alpha:T_\alpha(X^\perp)\to T_{\pi(\alpha)}(\mathbb PV)^\vee$
is not surjective;
\item[(3)] 
$\operatorname{Span}(T_\alpha(X^\perp),L_\alpha)
\subsetneq (\mathbb P^N)^\vee.$

\end{itemize}

\end{lemma}

\begin{proof}
On the open subset $(\mathbb P^N)^\vee\setminus \mathbb PV^\perp$, the linear
projection
\[
p:(\mathbb P^N)^\vee\dashrightarrow (\mathbb PV)^\vee
\]
is a morphism, and $\pi$ is its restriction to $X^\perp$.
The fiber of $p$ through $\alpha$ is the linear subspace
$L_\alpha=\operatorname{Span}(\alpha,\mathbb PV^\perp)$,
hence $\ker(dp_\alpha)=T_\alpha(L_\alpha)$.
Since $p$ is a linear projection, $dp_\alpha$ is surjective. Therefore the
restricted map
\[
d\pi_\alpha=dp_\alpha|_{T_\alpha(X^\perp)}
\]
fails to be surjective if and only if $T_\alpha(X^\perp)+\ker(dp_\alpha)
\neq T_\alpha((\mathbb P^N)^\vee)$,
that is,
\[
T_\alpha(X^\perp)+T_\alpha(L_\alpha)
\subsetneq T_\alpha((\mathbb P^N)^\vee).
\]
This is equivalent to the projective reformulation in (3).
\end{proof}

\noindent\textit{Proof of Proposition \ref{prop_critical-locus}.}
We will show for a general choice of sublinear system $V\subseteq H^0(\mathcal O_X(1))$, the closed set $(X|_V)^{\perp}$ is the closure of the locus where the tangent map $T\pi$ fails to be surjective.

We first show $\left(X\vert_V\right)^{\perp}\subseteq\mathcal{C} (\pi)$. By Lemma \ref{lemma_generalityV}, for general $\alpha\in\left(X\vert_V\right)^{\perp}\setminus \mathbb PV^\perp$, there exists $x\in X^{sm}\cap \mathbb PV$ with $T_xX\subseteq\alpha^{\perp}$. By classical projective duality \eqref{eqn_classical-proj-dual}, this is equivalent to $T_{\alpha}\left(X^{\perp}\right)\subseteq x^{\perp}$. Since $x\in \mathbb PV$, we have $\mathbb PV^\perp\subseteq x^\perp$, and since
$\alpha\in x^\perp$, it follows that
\[
L_\alpha=\operatorname{Span}(\alpha,\mathbb PV^\perp)\subseteq x^\perp.
\]
Hence
\[
T_\alpha(X^\perp)+T_\alpha(L_\alpha)\subseteq x^\perp,
\]
which is a proper hyperplane in $T_\alpha((\mathbb P^N)^\vee)$.
By Lemma \ref{lemma_C(pi)general}, we conclude that
$\alpha\in \mathcal C(\pi)$.


Conversely, we choose a general point $\alpha\in\mathcal{C}(\pi)\setminus \mathbb PV^\perp$, so we can assume $\alpha\in (X^{\perp})^{sm}$. By Lemma \ref{lemma_C(pi)general}, $T_\alpha(X^\perp)+T_\alpha(L_\alpha)$ is a proper subspace of $T_\alpha((\mathbb P^N)^\vee)$.
Therefore there exists a hyperplane
\[
W\subseteq (\mathbb P^N)^\vee
\]
containing both $T_\alpha(X^\perp)$ and $L_\alpha$.

Hence by projective duality (cf. Theorem \ref{thm_classicalPD}), $(\alpha,W^{\perp})\in N_{(\mathbb P^N)^{\perp}}^{\perp}(X^{\perp})\cong N_{\mathbb P^N}^{\perp}X$, so $W^{\perp}=x\in X$. On the other hand, since $\mathbb PV^{\perp}\subseteq L_{\alpha}\subseteq W$, by taking orthogonal complement we have $x\in L_{\alpha}^{\perp}\subseteq \mathbb PV$. Therefore $x\in X\cap \mathbb PV=X_V$, hence by definition \ref{def_critlocus}, we conclude that $\alpha\in (X|_V)^{\perp}$.  Therefore we conclude that $\mathcal{C}(\pi)\subseteq\left( X\vert_V\right)^{\perp}$.    \qed\\





\subsection{Proof of Theorem \ref{thm_mainthmequiv}.}

By Proposition \ref{prop_critical-locus}, we have
\[
\mathcal C(\pi)=(X|_V)^\perp.
\]
By Lemma \ref{lemma_maptodual}, the image of
$(X|_V)^\perp$ under $\pi$ is contained in $X_V^\perp$.
On the other hand, Proposition \ref{prop_critical-dom} shows that the induced map
\[
\pi|_{(X|_V)^\perp}:(X|_V)^\perp \dashrightarrow X_V^\perp
\]
is dominant. Therefore the critical values of $\pi$ contain a dense open subset of
$X_V^\perp$ and are contained in $X_V^\perp$. By definition of the smooth
discriminant,
\begin{equation*}
\Delta(\pi)=\overline{\pi(\mathcal C(\pi)\setminus \mathbb PV^\perp)}=X_V^\perp.
\end{equation*}
\qed

\begin{remark}\normalfont
An analogous statement of Proposition \ref{prop_critical-locus} was obtained in \cite[Proposition 3.9]{FT18} in the generic linear projection of the affine-analytic setting. In our projective setting (and as well Theorem \ref{thm_mainthmequiv}), the biduality isomorphism $N_{(\mathbb P^N)^{\perp}}^{\perp}(X^{\perp})\cong N_{\mathbb P^N}^{\perp}X$ is necessary in the proof. 
\end{remark}

\subsection{Relation to other work}
 In the case where the generic projection is non-dominant, Theorem \ref{thm_mainthmequiv} recovers the result in \cite[Theorem 1.21]{Tev05}.
\begin{corollary}[Non-dominant projection] \cite[Theorem 1.21]{Tev05}
     If $\dim(X^{\perp})<\dim(\mathbb PV^{\vee})$, i.e., the generic projection $\pi:X^{\perp}\to \mathbb PV^{\vee}$ is not dominant, then 
$$\pi(X^{\perp})=X_V^{\perp}.$$
     In other words, the image of the projection coincides with the dual variety of $X_V$ inside $\mathbb PV^{\vee}$.
\end{corollary}
\begin{proof}
    Since $\pi$ is non-dominant, the smooth critical locus $(X|_V)^{\perp}=X^{\perp}$ is the entire dual variety. In other words, every hyperplane $H$ tangent to $X$ at a smooth point must be also tangent to $X$ at a point contained in the linear section $X_V=X\cap \mathbb PV$. Hence by Definition \ref{def_discriminant}, $\pi((X|_V)^{\perp})=\pi(X^{\perp})$ is the smooth discriminant. In particular, Theorem \ref{thm_mainthmequiv} implies the equality $\pi(X^{\perp})=X_V^{\perp}$.
\end{proof}
When $X$ is a hypersurface, Theorem \ref{thm_mainthmequiv} recovers a result of Piene
\cite[Corollary 4.3]{Pie78}.

\section{Purity of smooth discriminant}\label{sec_purity}
In this section, we will prove Theorem \ref{Thm2proof}, which implies Theorem \ref{thm_Dirreducible-hypersurface}. Again, we assume $X\subseteq\mathbb P^N$ is an irreducible projective variety over an algebraically closed field with char $0$. We work with the notion of smooth discriminant (cf. Definition \ref{def_discriminant}) for generic projection of $X$. When the projection is finite and dominant, we prove Theorem \ref{thm_finite-normal}, where the discriminant is clean.

\subsection{Factorization of projections}
From now on, we switch back to the convention of Introduction, i.e., we take generic projection of $X$ and take linear section of $X^{\perp}$.

First, we show a direct consequence of Theorem \ref{thm_mainthmequiv}, which characterize the discriminant locus $D$ inductively as a discriminant locus of projection to a general hyperplane. This phenomenon was also observed in \cite[Proposition 2.2.5]{LeTeissier88}. One refers to \ref{sec_Veronese-surface} for an explicit example for this point of view.
\begin{lemma}\label{lemma_factor}
  Let $X\subseteq \mathbb P^N$ be an irreducible projective variety. Given a generic projection $p:\mathbb P^N\dashrightarrow \mathbb P^k$, we factor it as a composite of projections to the hyperplanes, as shown in the first row of the following commutative diagram:
\begin{figure}[ht]
    \centering
\begin{equation}
\begin{tikzcd}
\mathbb P^N\arrow[r,dashrightarrow]&\mathbb P^{N-1}\arrow[r,dashrightarrow]&\cdots \arrow[r,dashrightarrow]&\mathbb P^{k+1}\arrow[r,dashrightarrow]&\mathbb P^{k}\\
X\arrow[u,symbol=\subseteq]\arrow[ur,"\pi_N"']& \Delta_{N-1}\arrow[u,symbol=\subseteq]\arrow[ur,"\pi_{N-1}"']&\cdots& \Delta_{k+1}\arrow[u,symbol=\subseteq]\arrow[ur,"\pi_{k+1}"']&\Delta_k\arrow[u,symbol=\subseteq].
\end{tikzcd}
\end{equation}
\end{figure}
    
    Let $\Delta_N=X$, we inductively define $\Delta_l\subseteq\mathbb P^l$ to be the smooth discriminant locus of the generic projection of $\pi_{l+1}:\Delta_{l+1}\to\mathbb P^l$ (which is a finite branched cover). Then the smooth discriminant locus $\Delta$ of $\pi|_X:X\dashrightarrow \mathbb P^k$ coincides with $\Delta_k$ as subvarieties of $\mathbb P^k$.
\end{lemma}
\begin{proof}
We note that the sequence of generic projections 
$$\mathbb P^N\dashrightarrow\mathbb P^{N-1}\dashrightarrow\cdots\dashrightarrow\mathbb P^{k+1}\dashrightarrow\mathbb P^k$$
corresponds to hyperplanes $H_1,\ldots,H_{N-k}$ in $\mathbb P^N$, such that for each $1\le l\le N-k$, there is a canonical identification
$$(\mathbb P^{N-l})^{\vee}=H_1\cap\cdots \cap H_l.$$
Then apply Theorem \ref{thm_mainthmequiv} inductively to each $1\le l\le N-k$. Under the natural identifications, we obtain 
$$\Delta_{N-l}^{\perp}\cong X^{\perp}\cap (H_1\cap  \cdots \cap H_l).$$
In particular, when $i=N-k$, we prove the claim.
\end{proof}

\begin{corollary} \label{cor_factor2}
Let $X\subseteq \mathbb P^N$ be a $n$-dimensional irreducible projective variety. Then any generic projection can be factored as 
$$X\looparrowright \mathbb P^{n+1}\dashrightarrow \mathbb P^k,$$
where the first map $f$ is a generic immersion  onto a hypersurface $X'$ of $\mathbb P^{n+1}$, and a generic projection $g$ of $X'$, whose smooth discriminant is precisely $\Delta$. 
\end{corollary}
\begin{proof} From Lemma \ref{lemma_factor}, $f$ is obtained by composing the first $N-n-1$ maps. Since $f$ is a generic immersion, $df$ is injective on a dense open subset of $X$. Hence the smooth critical locus of $\pi$ agrees birationally with the smooth critical locus of $g:X'\dashrightarrow \mathbb P^k$ on a dense open subset.   Therefore, the claim follows.
\end{proof}


\subsection{Proof of purity theorem}

Now, Theorem \ref{thm_Dirreducible-hypersurface} will follow from the following and Lemma \ref{lemmma_disc_vs_smoothdisc}:

\begin{theorem}[cf.  Theorem \ref{thm_Dirreducible-hypersurface}]\label{Thm2proof}
 Let $X\subset \mathbb P^N$ be an irreducible projective variety. Assume that the generic projection 
  $$X\dashrightarrow \mathbb P^k$$
  is dominant (equivalently, $\dim(X)\ge k$). Then the smooth discriminant locus $\Delta$ is either 
    \begin{itemize}
    \item[(i)] an irreducible hypersurface, if $\textup{codim}(X^{\perp})<k$, or
    \item[(ii)]  union of hyperplanes, if $\textup{codim}(X^{\perp})=k$, or
    \item[(iii)]  empty, if $\textup{codim}(X^{\perp})>k$.
    \end{itemize}
\end{theorem}
\begin{proof}
    If $\textup{codim}(X^{\perp})>k$, $X^{\perp}$ does not intersect with a general linear space of dimension $k$, hence by Theorem \ref{thm_dual}, the discriminant locus is empty; If $\textup{codim}(X^{\perp})=k$, the intersection $X^{\perp}$ with a general $\mathbb P^k$ consists of finitely many points, whose dual variety is a union of hyperplanes; If $\textup{codim}(X^{\perp})<k$, then Bertini theorem implies $\mathbb P^k\cap X^{\perp}$ is irreducible, hence by Theorem \ref{thm_dual}, its dual variety, the smooth discriminant $\Delta$ is irreducible as well.  

    It remains to show $\Delta$ has codimension one under the assumption that $\textup{codim}(X^{\perp})<k$. Since $n:=\dim(X)\ge k$, we can factor the generic projection as a sequence of generic projections to hyperplanes, and we have a sequence of generic projections:
    $$X\xrightarrow{\pi_{n+1}} \mathbb P^{n}\supseteq \Delta_n\xrightarrow{\pi_n} \mathbb P^{n-1}\supseteq \Delta_{n-1}\to \cdots \supseteq \Delta_{k+1}\xrightarrow{\pi_{k+1}} \mathbb P^k,$$
    where each $\Delta_l$ is defined as the smooth discriminant locus of the finite branched cover $\pi_{l+1}:\Delta_{l+1}\to \mathbb P^l$. $\Delta_l$ is nonempty by our assumption and Lemma \ref{lemma_factor}.  Also, notice that by taking normalization $\tilde{\Delta}_{l+1}$ of $\Delta_{l+1}$. The finite branched cover 
    $$\tilde{\pi}_{l+1}:\tilde{\Delta}_{l+1}\to \Delta_{l+1}\to \mathbb P^{l}$$
    has the same smooth discriminant: $\Delta(\tilde{\pi}_{l+1})=\Delta_l$. Now, since $\Delta_l$ is an irreducible component of the branch loci of $\tilde{\pi}_{l+1}$,  Zariski--Nagata purity Theorem \ref{thm_Zariski-Nagata} implies that each $\Delta_l\subseteq \mathbb P^l$ is a hypersurface. Finally, by Lemma \ref{lemma_factor}, $\Delta_k$ coincides with the smooth discriminant $\Delta$, thus the claim follows.
\end{proof}


\begin{remark}[Genericity of the projection center]\normalfont \label{remark_proj-from-p-on-Q}
When generic projection $\pi:X\dashrightarrow \mathbb P^k$ has positive dimensional fiber, the purity relies on the genericity of the projection center. For example, projecting a quadric surface $Q$ from a point $p$ lying on it yields a birational map $Q\dashrightarrow \mathbb P^2$ contracting the two lines through $p$. Therefore, the discriminant locus consists of two points, which has codimension two.

\end{remark}






\subsection{Purity of smooth critical locus}
In the rest of this section, we prove an analogous purity statement for the smooth critical locus of a dominant generic projection. 

We first strengthen Proposition \ref{prop_critical-dom}:

\begin{proposition}\label{prop_one-to-one}
 With the same assumption as before, the map  $\pi|_{(X|_V)^{\perp}}$ is generically one-to-one. In particular, for generic projection, the map from the smooth critical locus to the smooth discriminant locus
 $$\mathcal {C}(\pi)\to \Delta(\pi)$$
 is birational.
\end{proposition}

\begin{proof}
 By Corollary \ref{cor_factor2}, this reduces to the case where $X^{\perp}$ is a hypersurface. Then it has a unique tangent hyperplane at a general point. Therefore, \eqref{eqn_section} defines the unique lifting of $\pi|_{(X|_V)^{\perp}}$ at a general point, so it is generically one-to-one.
\end{proof}

\begin{corollary}
    The smooth critical locus of a generic projection (cf. \eqref{intro_eqn_genproj}) either has pure codimension $r+1$, where $r=\dim(X)-k$ is the relative dimension, or is empty.
\end{corollary}
\begin{proof}
    This follows from Theorem \ref{Thm2proof} and Proposition \ref{prop_one-to-one} that the smooth critical locus is birational to the smooth discriminant locus.
\end{proof}
We will continue the discussion of the smooth critical locus in Section \ref{sec_polar}.

\subsection{Finite branched cover}
Now we specialize Theorem \ref{thm_mainthmequiv} to the case where the generic projection is finite and dominant. Recall that in this case, the smooth discriminant coincides with Fitting discriminant (See Lemma \ref{lemma_finite-clean}).

\begin{theorem}\label{thm_finite-normal}
     Let $X\subseteq \mathbb P^N$ be a normal projective variety of dimension $n$. Let $$\pi:X\to \mathbb P^n$$ be a generic projection. Note that this is regular and is a finite dominant morphism. Then the branch divisor $\Delta=\Delta_{Fitt}$ is the discriminant of $\pi$. Moreover,
     \begin{itemize}
         \item[(1)]  there is an equality of projective subvarieties of $(\mathbb P^n)^{\vee}$:
     $$ \Delta^{\perp}= X^{\perp}\cap (\mathbb P^n)^{\vee}.$$ 
     \item[(2)] $\Delta$ is a hypersurface iff $\textup{codim}(X^{\perp})<n$, a union of hyperplanes iff $\textup{codim}(X^{\perp})=n$, and empty iff $\textup{codim}(X^{\perp})>n$.
     \end{itemize}

\end{theorem}
\begin{proof}
    This follows from Lemma \ref{lemma_finite-clean}, Theorem \ref{thm_mainthmequiv} and \ref{Thm2proof}.
\end{proof}

\section{Dual defective varieties}\label{sec_emptydisc}
In this section, we discuss some consequences of Theorem \ref{Thm2proof}. We specifically focus on case (ii) and (iii) of Theorem \ref{Thm2proof} and prove Corollary \ref{intro_cor_smoothfib}.

\subsection{Empty smooth discriminant}\label{sec_smoothfib}

In this section, we discuss case (iii) of Theorem \ref{Thm2proof} and investigate the relationship with dual defective varieties (see Remark \ref{remark_defect}).

\begin{corollary}\label{cor_smooth-fibration}
Let $X \subseteq \mathbb{P}^N$ be a projective variety such that $\mathbb{P}V \cap X^\perp = \emptyset$, and let $B$ be the base locus of the projection $\pi: X \dashrightarrow \mathbb{P}V^\vee$. Then the induced morphism $$\tilde{\pi}: \textup{Bl}_BX \to \mathbb{P}V^\vee$$ is smooth on $(\textup{Bl}_BX)^{sm}$, the smooth locus of $\textup{Bl}_BX$. 

When $X$ is smooth, then $\tilde{\pi}$ is a smooth proper morphism if and only if $\textup{codim}(X^{\perp })>\dim(\mathbb PV)$.
\end{corollary}
\begin{proof}
    According to the proof of Theorem \ref{thm_mainthmequiv}, $\mathbb PV\cap X^{\perp}$ being empty implies that the critical locus of $(\textup{Bl}_BX)^{sm}\to \mathbb PV^{\vee}$ is empty. Hence if $X$ is smooth, so is $\textup{Bl}_BX$, and the morphism $\tilde{\pi}$ is smooth. 
\end{proof}


As an application, this recovers a classical result that dual defect does not occur in $\dim(X)=1$ and $2$.

\begin{proposition}[Dual defect for curve and surface] \cite[Proposition 3.1]{Ein86}\label{prop_no-defect-dim1,2}
   Let $X\subseteq \mathbb P^N$ be a smooth projective variety of dimension $1$ or $2$. Suppose $X$ is not a linearly embedded line or plane, then the dual variety $X^{\perp}$ is a hypersurface.
\end{proposition}
\begin{proof}
 Suppose $X^{\perp}$ is not a hypersurface. If $\dim(X)=1$, then by Corollary \ref{cor_smooth-fibration}, a generic projection $X\to \mathbb P^1$ is unramified, which forces $X\cong \mathbb P^1$ to be a linearly embedded line; If $\dim(X)=2$, then a generic projection to $\mathbb P^1$ extends to a smooth family of curves
    $$\tilde{X}=\textup{Bl}_BX\to \mathbb P^1.$$
Then $\pi$ must be isotrivial, because of hyperbolicity of $M_g$ for $g\ge 2$, and by classification of elliptic surface for $g=1$. Hence $\tilde{X}$ is either a Hirzebruch surface or $F\times \mathbb P^1$ for a curve $F$ with $g(F)\ge 1$. They are both minimal surfaces, contradicting the fact that $\tilde{X}$ arises as a nontrivial blowup of $X$, unless $\tilde{X}$ is the first Hirzebruch surface $\mathbb F_1$, which arises from blowing up a linearly embedded $\mathbb P^2$ at one point. Hence, the claim follows.
    \end{proof}
\begin{example}[Defective 3-fold dual] \normalfont\label{example_segre3fold}
     The first example of the defective dual variety starts from dimension three: the Segre 3-fold $X=\mathbb P^1\times \mathbb P^2\hookrightarrow \mathbb P^5$ has dual defect one \cite{Ein86}. By Corollary \ref{cor_smooth-fibration}, the blow-up of $X$ along the base locus $B$, which is a twisted cubic, admits a smooth morphism
     $$\textup{Bl}_BX\to \mathbb P^1.$$
     Every fiber is isomorphic to the first Hirzebruch surface $\mathbb F_1$.

\end{example}

\begin{remark}[Singular total space]\normalfont \label{remark_twistedcubicdual}
  When $X$ is singular (particularly non-normal), then $\tilde{\pi}$ in Corollary \ref{cor_smooth-fibration} may not be smooth everywhere, 
    even passing to an appropriate desingularization. This is because the singular strata of $X$ may not be smooth over the base, and this is not captured by the Definition \ref{def_discriminant}. See an example below.
 
\end{remark}
\begin{example}[dual of twisted cubic]\normalfont\label{example_dual-twisted-cubic}
        We take $X$ to be the dual variety of a twisted cubic $C\cong X^{\perp}\subseteq \mathbb P^3$.     Since a general $\mathbb P^1$ will miss the curve $C$, i.e., $C\cap \mathbb P^1=\emptyset$, Theorem \ref{thm_dual} implies that the generic projection $X\dashrightarrow\mathbb P^1$ has an empty smooth discriminant. However, the morphism by blowing up the base locus $B$
$$\tilde{\pi}:\textup{Bl}_{B}X\to \mathbb P^1$$
is not smooth, even after passing to a suitable desingularization. In fact,  $X$ is a quartic surface and has transversal cuspidal singularities along the osculating curve $C'$. The curve $C'$ is isomorphic to the twisted cubic $C$ itself. Hence $B$ consists of four points and the general fiber of $\tilde{\pi}$ is a quartic curve with three cusps. The family $\tilde{\pi}$ is isotrivial, but has four reducible fibers corresponding to the four branch points of $C'\to \mathbb P^1$. The normalization $\widetilde{\Bl_BX}$ of $\Bl_BX$ is a weak del Pezzo surface of degree four, and the composition map 
$$\widetilde{\Bl_BX}\to \Bl_BX\to \mathbb P^1$$
has general fiber $\mathbb P^1$, but contains 4 reducible fibers.
\end{example}



\subsection{Reducible smooth discriminant} \label{sec_reducible-disc} In this subsection,  we give an example of case (ii) of Theorem \ref{Thm2proof}. Note that if $X$ is smooth, then generic projection $X\dashrightarrow \mathbb P^k$ has reducible discriminant with $k\ge 2$ only occurs when $X$ is dual defective (see Corollary \ref{cor_smooth-fibration}). As we have seen in Proposition \ref{prop_no-defect-dim1,2}, nontrivial examples start from dimension three. We give a simple example where the total space is singular and is a two-dimensional total.

\begin{example}[dual of spatial curves]\normalfont \label{example_dualspatialcurve}
    Let $C\subseteq \mathbb P^3$ be a projective curve of degree $d$, and let $X=C^{\perp}$ be the dual variety (also see Example \ref{example_dual-twisted-cubic}). Then $X$ is a surface in the dual space $(\mathbb P^3)^{\vee}$. It is singular along a curve $C'$, called the osculating curve associated with $C$. Choosing a general plane $\mathbb P^2\subseteq \mathbb P^3$, then the projection
    $$X\to (\mathbb P^2)^{\vee}$$
    is a branched cover, and its smooth discriminant consists of $d$ lines, which is the dual of $\mathbb P^2\cap C=d$ points, by Theorem \ref{thm_mainthmequiv}. The converse also holds—If the generic projection of a surface $X\subseteq \mathbb P^3$ to $\mathbb P^2$ has smooth discriminant being the union of lines, then the dual variety of $X$ is a curve. This is compatible with the fact that dual variety $C^{\perp}$ is the tangential variety of the osculating curve $C'$ \cite[p.640]{Dolgachev-classicalAG}.
\end{example}

\section{Examples and Applications}\label{sec_example}
In this section, we illustrate the applications of the main theorems \ref{thm_dual} and \ref{thm_Dirreducible-hypersurface} in three examples involving surfaces and threefolds. In Section \ref{sec_Veronese-surface}, we take successive generic projections of Veronese surface to hyperplanes. Correspondingly in the dual space, we find special semistable cubic hypersurfaces that are related by taking successive general hyperplane sections; In Section \ref{subsec_nodes&cusps}, we compute the number of nodes and cusps of the discriminant curve arising from the generic projection of a smooth projective surface and of the Veronese threefold.

\subsection{Veronese surface and sections of cubics} \label{sec_Veronese-surface}

 Let $v_2(\mathbb P^2)$ be the Veronese image $v_2:\mathbb P^2\hookrightarrow \mathbb P^5$. Then its dual variety $v_2(\mathbb P^2)^{\perp}$ is the discriminant hypersurface parameterizing singular planar conics \cite[Sec. 2.1.1]{Dolgachev-classicalAG}. Moreover, it is a cubic fourfold singular along a smooth surface isomorphic to the Veronese surface $v_2(\mathbb P^2)$. We denote this dual variety by $M_4$. Define $M_i=M_{i+1}\cap H$ inductively a general hyperplane section of $M_{i+1}$ for $i\le 3$. 
 
 Then Theorem \ref{thm_dual} implies that $M_i$ is projective dual to the discriminant of $v_2(\mathbb P^2)\to \mathbb P^{i+1}$. We summarize these projective dualities in the following table:
 

   

\begin{table}[h!]
\begin{center}
\begin{tabular}{|c|c|c|c|} 
 \hline
Discriminant loci of generic projections of $v_2(\mathbb P^2)$ & Hyperplane sections of $M_4:=v_2(\mathbb P^2)^{\perp}$ \\ \hline\hline
$v_2(\mathbb P^2)\hookrightarrow\mathbb P^5$ Veronese surface  & $M_4=$ Cubic fourfold singular along $v_2(\mathbb P^2)$ \\ \hline
$v_2(\mathbb P^2)\hookrightarrow \mathbb P^4$ smooth image &  $M_3=$ Cubic threefold singular along $C_4$ \\ \hline
$v_2(\mathbb P^2)\looparrowright \mathbb P^3$ Steiner quartic surface& $M_2=$ Cubic surface with 4$A_1$ singularities \\ \ChangeRT{1.5pt}
$v_2(\mathbb P^2)\xrightarrow{4:1} \mathbb P^2$ branched along a sextic curve& $M_1=$ smooth cubic curve  \\ \hline
$v_2(\mathbb P^2)\dashrightarrow\mathbb P^1$ quartic fibration & $M_0=$ 3 points  \\ \hline
\end{tabular}
\caption{Projective duality between discriminant locus of generic projections of the Veronese surface $v_2(\mathbb P^2)$ and successive hyperplane sections of dual variety $v_2(\mathbb P^2)^{\perp}$. }
\label{table}
\end{center}
\end{table}
The first three rows correspond to non-dominant projections, so by our convention the discriminant is the projection image. $C_4$ is the rational normal quartic, as a hyperplane section of $v_2(\mathbb P^2)$. The fourth row corresponds to a dominant finite projection, whose discriminant is the branch locus. The last row corresponds to a Lefschetz fibration, whose discriminant is the set of critical values.

The computation uses the fact that the Veronese surface is a Severi variety \cite[Theorem 2.4]{Zak_tangents&secants}. These cubic hypersurfaces also appear in moduli theory \cite{CMSHL21,Loo09}: $M_3$ (resp. $M_4$) is strictly semistable in the GIT compactification of cubic threefolds (resp. fourfolds). For $2\le i\le 4$, the cubic hypersurface $M_i$ is unique up to projective equivalence, and corresponds under projective duality to the image of a non-dominant generic projection of $v_2(\mathbb P^2)$.


\subsection{Discriminant curves and their singularities}\label{subsec_nodes&cusps}
In this section, we study the discriminant curve of a generic projection of a smooth projective variety to a plane
$$X\dashrightarrow \mathbb P^2.$$
When $X$ is not dual defective, then the discriminant locus $\Delta$ is an irreducible curve (cf. Theorem \ref{thm_Dirreducible-hypersurface}). The singularities $\Delta$ are expected to only have ordinary nodes and cusps. Under this assumption, the numbers of nodes $\delta$ and cusps $\kappa$ can be explicitly deduced from the duality theorem \ref{thm_dual} and Pl\"ucker formulas. 

In the case when $X$ is a smooth projective surface and the Veronese 3-fold, it is known $\Delta$ only have nodes and cusps. We demonstrate how to find $\delta$ and $\kappa$ below. 

\subsubsection{Generic projection of surfaces}
We revisit the discriminant invariants of generic projection of a surface computed in \cite{CMT01}, which are used to describe a special type III degeneration of K3 surfaces. Here we give a different proof of the results using projective duality.

\begin{theorem}\cite{CMT01, CilFla11}
    Let $S\subseteq \mathbb P^N$ be a smooth surface and $\pi:S\to \mathbb P^2$ a generic projection. Then its discriminant curve $\Delta$ has only nodes and cusps. Its degree, the number of nodes $ \delta$ and cusps $\kappa$ equal to 
\begin{align}
\deg(\Delta)&=3\deg(S)+K_S\cdot H,\\
   \delta&=\frac12\deg(\Delta)^2-15\deg(\Delta)-3K_S^2+24\deg (S)+c_2(S),\\
    \kappa&=9\deg(\Delta)+2K_S^2-15\deg(S)-c_2(S).  
\end{align}
\end{theorem}
\begin{proof}The branch cover $S\to \mathbb P^2$ has ramification locus $\tilde{\Delta}$. In \cite{CilFla11}, it shows that the map $\tilde{\Delta}\to \Delta$ is a desingularization, and the discriminant curve $\Delta$ has only nodes and cusps as its singularities. Here we take this fact.

The degree of $\Delta$ follows from the adjunction formula $K_S=\pi^*K_{\mathbb P^2}+\tilde{\Delta}$, where $\tilde{\Delta}\subseteq S$ is the ramification divisor, and the fact that $\pi^*K_{\mathbb P^2}=\mathcal{O}_S(-3)$.

   Theorem \ref{thm_mainthmequiv} implies that the dual curve $\Delta^{\perp}$ of $\Delta$ is a plane section $S^{\perp}\cap (\mathbb P^2)^{\vee}$ of the dual variety of $S$. Therefore, we obtain the degree of $\Delta^{\perp}$
$$d^{\perp}=\deg(S^{\perp})=c_2(S)+2K_S\cdot H+3\deg(S)$$
from the formula of degree of dual variety \cite[IV,64]{Kle77}. Since dual curves are birational and have the same geometric genus, we apply the adjunction formula to $\tilde{\Delta}\subseteq S$ and find the genus $g^{\perp}$ of $\Delta^{\perp}$:
$$g^{\perp}=g(\tilde{\Delta})=\frac12[2K_S^2+9K_S\cdot H+9\deg(S)]+1.$$ 
Finally, the number of nodes and cusps follow from the Pl\"ucker formulas \eqref{eqn_Plucker1}, \eqref{eqn_Plucker2}, and knowledge of $d^{\perp},g^{\perp}$, and $\deg(\Delta)$.
\end{proof}

\begin{example}[Surface in $\mathbb P^3$]\label{prop_S->P^2}\normalfont
 When $S_d\subseteq \mathbb P^3$ is a smooth surface of degree $d$, then the generic projection $S_d\to \mathbb P^2$
has branch curve $\Delta$ with degree $$\deg(\Delta)=d(d-1).$$ 
The number of nodes and cusps of $\Delta$ equal to $$\delta={1\over 2}d(d-1)(d-2)(d-3),\ \kappa=d(d-1)(d-2),$$ respectively. Specifically in low degrees:
\begin{itemize}
    \item[$d=2$] A quadric surface $S_2\to \mathbb P^2$ is a double cover branched along a smooth conic (compare with Remark \ref{remark_proj-from-p-on-Q});
    \item[$d=3$] A cubic surface $S_3\to \mathbb P^2$ is a triple cover branched along a plane sextic curve with 6 cusps. (In fact, one can show that the 6 cusps lie on a conic, so these sextic curves lie in a distinguished component of Zariski pair \cite{Zariskipair}.)
\end{itemize}
\end{example}

\subsubsection{A net of quadrics and an elliptic threefold}\label{sec_sub_ell-3fold}

\begin{example}[Net of quadrics]\normalfont\label{example_MWrank7} 
Let $X=v_2(\mathbb P^3)$ be the image of the second Veronese embedding of $\mathbb P^3$. Then the generic projection 
  \begin{equation}\label{intro_eqn_v2(P^3)}
     v_2(\mathbb P^3)\dashrightarrow \mathbb P^2 
    \end{equation}
has general fiber being a genus one curve as a complete intersection of two quadrics. The base locus is the complete intersection of three general quadrics and consists of 8 points, which is called a Cayley octad (see \cite{CayleyOctad18}). 

We characterize the discriminant locus below.
\begin{proposition}
   The generic projection \eqref{intro_eqn_v2(P^3)} extends to an elliptic threefold
\begin{equation}\label{eqn_ell3fold}
    \tilde{\pi}:\widetilde{v_2(\mathbb P^3)}\to \mathbb P^2,
\end{equation}
whose discriminant $\Delta$ is an irreducible curve of degree $12$, and parameterizes singular fibers described below: 
\begin{itemize}
    \item nodal fibers over smooth points of $\Delta$; 
    \item $28$ Kodaira type $I_2$ fibers (union of two conics meeting at two points) over the nodes of $\Delta$;
    \item $24$ cuspidal fibers over the cusps of $\Delta$.
\end{itemize} 

\end{proposition}

\begin{figure}[h]
\centering
\includegraphics[width=0.5\textwidth]{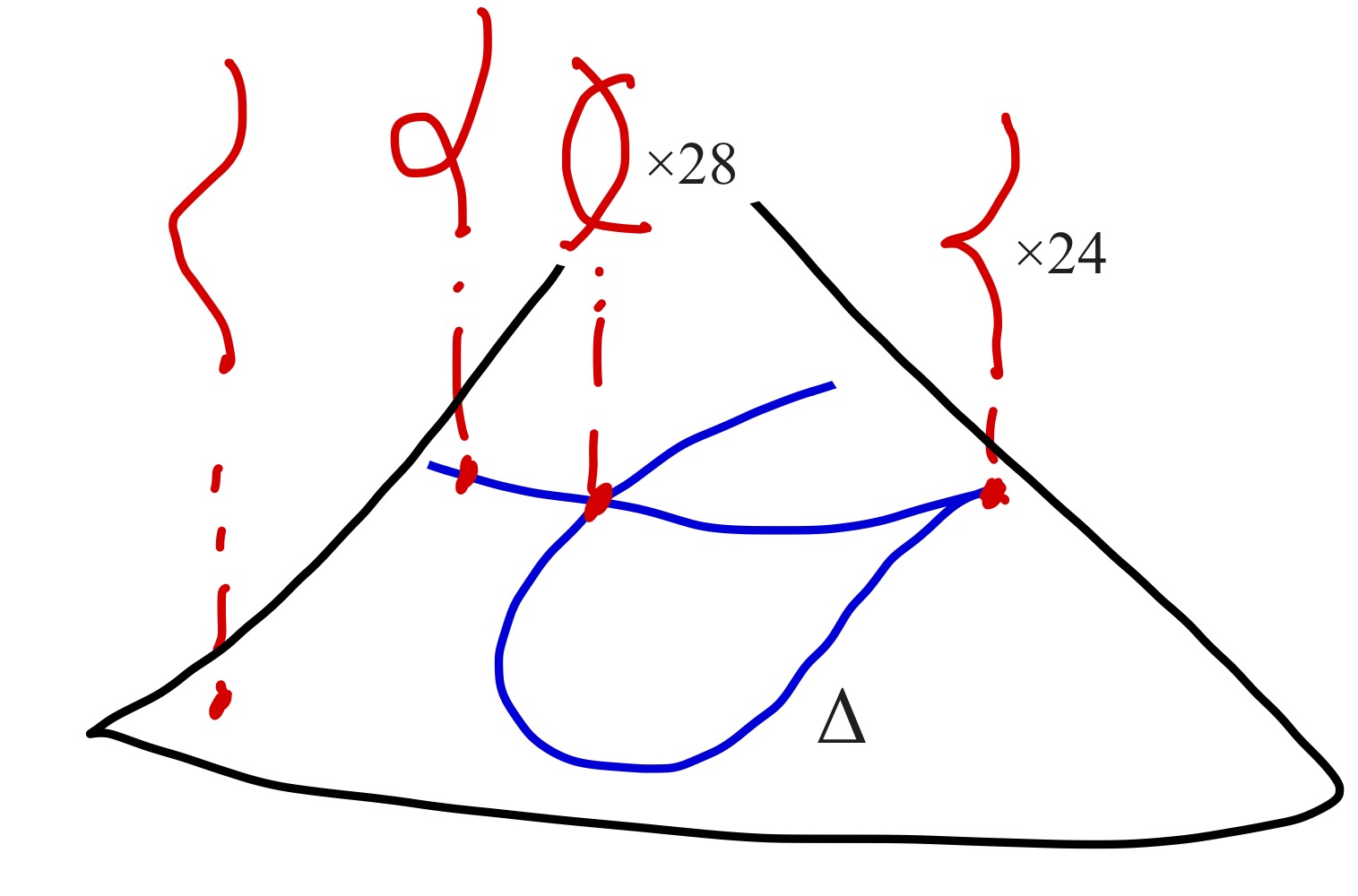}
\caption{Discriminant of generic projection of the Veronese threefold}
\end{figure}

We give a sketch of the proof. We note that the "dual family" of \eqref{intro_eqn_v2(P^3)} is a general net of quadric surfaces, namely, a general two dimensional family of the hyperplane sections of $v_2(\mathbb P^3)$. The discriminant $\Delta^{\perp}$ is a general plane section of the dual variety $v_2(\mathbb P^3)^{\perp}$. By classical invariant theory  \cite[p.16]{Dolgachev-classicalAG}, this dual variety parameterizes degenerate quadrics on $\mathbb P^3$. Equivalently, it is isomorphic to the space of degenerate $4\times 4$ symmetric matrices:
\begin{equation*}\label{eqn_quartic-hypersurface}
    v_2(\mathbb P^3)^{\perp}\cong \{A\in Mat_{4\times 4}|A^T=A,\ \det(A)=0\}.
\end{equation*}
 In particular, $v_2(\mathbb P^3)^{\perp}$ is a normal quartic hypersurface. Hence $\Delta^{\perp}$ is a smooth plane quartic curve, and \cite{Ter10} guarantees that it is general in the moduli. Therefore, by classical algebraic geometry, $\Delta^{\perp}$ has 28 bitangents and 24 flex lines. The rest follows from Theorem \ref{thm_dual} and Pl\"ucker formulas \eqref{eqn_Plucker1}, \eqref{eqn_Plucker2}.
 \end{example}

Recall that the Mordell--Weil group of an elliptic fibration is the group of rational sections, whose rank is called the Mordell--Weil rank. We recover the computation of the Mordell--Weil rank of the elliptic 3-fold \eqref{eqn_ell3fold} below, which is first obtained in \cite[Theorem 7.2]{Tot08}.
\begin{proposition}\cite{Tot08}
   The Mordell--Weil rank of the elliptic threefold \eqref{eqn_ell3fold} is  7. 
\end{proposition}
\begin{proof}
  By construction, the elliptic threefold $\widetilde{v_2(\mathbb P^3)}$ has Picard rank 9.
   The discriminant locus is irreducible, and a general fiber over it is nodal curve, so by Shioda-Tate-Wazir formula \cite{Wazir04}, we have
$$\rho(\widetilde{v_2(\mathbb P^3)})=1+\rho(\mathbb P^2)+r+0,$$
where $r$ is the Mordell--Weil rank, hence $r=7$ as claimed.
\end{proof}
Note that there is an interest in constructing an elliptic threefold with a large Mordell--Weil rank over $\mathbb P^2$ (or rational surfaces), see \cite{CogoLib14,HulekKloo11}. The current record is 10 \cite{GrassiWei22}, but it is isotrivial. The above seems to give the largest Mordell--Weil rank (among known examples) for non-isotrivial ones.

\begin{remark}\normalfont
  By Theorem \ref{thm_dual}, the discriminant curve $\Delta$ of $X\dashrightarrow \mathbb P^2$ is the dual curve of a general plane section of the dual variety $\Delta^{\perp}=(\mathbb P^2)^{\vee}\cap X^{\perp}$. Suppose $\Delta^{\perp}$ has only nodes and cusps, then their numbers $\delta^{\perp}$ and $\kappa^{\perp}$ are computed in \cite[(2.2), (2.3)]{DolLib81}, \cite{Roberts}. Then $g(\Delta^{\perp})=g(\Delta)$, $\deg(\Delta^{\perp})$, $\delta^{\perp}$ and $\kappa^{\perp}$ are related to $\deg(\Delta)$, $\delta$, $\kappa$ via the Pl\"ucker formulas. In general, this method can characterize codimension-one singularities of discriminant $\Delta$ for generic projection $X\dashrightarrow \mathbb P^k$ for any $k$.
\end{remark}


\section{Polar locus and degree of smooth discriminant}\label{sec_polar}

In this section, we show that the smooth critical locus of a generic projection (cf. Definition \ref{def_smoothcrit_rat}) coincides with the notion of polar locus in the literature. As a consequence, the degree of smooth discriminant admits an intersection theory formula. We recover some classical identities on the degree of polar loci.  

\subsection{Polar locus}
Let $X\subseteq \mathbb P^{N}$ be an irreducible projective variety. Recall that $N^{\perp}_{\mathbb P^N}(X)$ is the conormal variety \eqref{eqn_conormal}. It has dimension $N-1$ and admits two projections: the first projection $\pi_1$ to $X$ is generically a projective bundle, and the second projection $\pi_2$ has image the dual variety $X^{\perp}$. Following \cite[Cha. II]{Kle86}, for each $0\le i\le N$, one defines the intersection number
\begin{equation}\label{eqn_r_i}
    r_i(X)=\int_{N^{\perp}_{\mathbb P^N}(X)}\pi_1^*H^i\cdot \pi_2^*(H')^{N-1-i},
\end{equation}
 where $H=c_1\mathcal{O}_{\mathbb P^N}(1)$ and $H'=c_1\mathcal{O}_{(\mathbb P^N)^{\vee}}(1)$ are hyperplane classes on $\mathbb P^N$ and the dual space $(\mathbb P^N)^{\vee}$, respectively. 
 
 These intersection numbers have special geometric meanings. For example, we assume $X$ is irreducible and denote $n=\dim(X)$ as usual, then note that by the projection formula, $r_n=\deg(X)$ is the degree of $X$; $r_i=0$ when $i>n$; and 
 \begin{equation*}
r_0=\begin{cases}\deg(X^{\perp}), \ \textup{if}\  X^{\perp}\ \textup{is a hypersurface};\\
0,\ \textup{otherwise}.
\end{cases}.    
 \end{equation*}
 
  In general, $r_i$ counts the number of $N-1-i$ planes that are tangent to $X$ in a general linear pencil tangent to $X$.  
  
  \begin{definition}\textup{\cite[p. 186]{Kle86}}\normalfont 
  \ The \textit{$i$-th polar locus} of $X$ with respect to a general codimension $i+2$ linear subspace $A_i\subseteq \mathbb P^N$ is defined to be 
$$X(A_i)=\overline{\{p\in X^{sm}|\dim T_pX\cap A=n-i-1\}},$$
consisting of points where the corresponding tangent plane intersects $A_i$ having larger dimension than expected. 
  \end{definition}
 The intersection number \eqref{eqn_r_i} has the following equivalent characterization:
\begin{proposition}\cite[p.187]{Kle86} For $0\le i\le n$, the intersection number $r_i$ coincides with the the degree of $i$-th polar locus $$r_i=\deg(X(A_i)).$$
\end{proposition}
  In general, one can talk about the class of $i$-th polar locus $[X(A_i)]$ in the Chow group. There is an enumerative geometry formula to compute this class \cite[IV,21]{Kle77}. 
\subsection{Relation to generic projection}
We will give an equivalent characterization from the point of view of generic projection. Note that the subspace $A_i\subseteq \mathbb P^{N}$ induces a rational map $\mathbb P^N\dashrightarrow \mathbb P^{i+1}$, which restricts to $X$ is a generic projection
$$\pi: X\dashrightarrow \mathbb P^{i+1},$$
as long as $A_i$ is in general position. We observe that

\begin{lemma}\label{lemma_polar=crit}
The polar locus $X(A_i)$ coincides with the smooth critical locus of $\pi$.    
\end{lemma}
\begin{proof}
    Note that the projection $\mathbb P^N\dashrightarrow \mathbb P^{i+1}$ coincides with the dual of the inclusion $A_i^{\perp }\subseteq (\mathbb P^N)^{\vee}$. So by
    Definition \ref{def_critlocus}, \ref{def_smoothcrit_rat} (and switch $X$ and its dual) and Proposition \ref{prop_critical-locus}, the smooth critical locus of $\pi$ is $\pi_1(\pi_2^{-1}(A_i^{\perp}\cap X^{\perp}))$. The general point of this set consists of $x\in X$, such that there exists a hyperplane $H_{\alpha}$ satisfying $T_xX\subseteq H_{\alpha}$ and $A_i\subseteq H_{\alpha}$. Hence $T_xX$ and $A_i$ are collinear and $\dim (T_xX\cap A_i)\ge n-i-1$, so $x$ is contained in the $i$-th polar locus $X(A_i)$, and vice versa.
\end{proof}

Recall that we have showed that via the generic projection $\pi$, the smooth critical locus is generically one-to-one onto the smooth discriminant (cf. Proposition \ref{prop_one-to-one}), Lemma \ref{lemma_polar=crit} implies that
\begin{corollary}\label{cor_r_i=deg(disc)}
  For each $0\le i\le n$, the intersection number $r_i$ \eqref{eqn_r_i} is the degree of the smooth discriminant locus $\Delta_{i+1}$ of the generic projection $\pi: X\dashrightarrow \mathbb P^{i+1}.$ In other words, we have
$$\deg(\Delta_{i+1})=r_i(X).$$
\end{corollary}


\subsection{Classical results}
As a direct consequence of the above, we recover various classical results on the degree of the polar locus, which by Corollary \ref{cor_r_i=deg(disc)} also implies the relations between $\deg(\Delta_{i+1})$ for various $i$.

\begin{theorem}\cite{HK85}
    $r_i\neq 0$ iff $\ddef(X)\le i\le n$, where $\ddef(X)$ is the defect of $X^{\perp}$ (cf. Definition \ref{remark_defect}).
\end{theorem}
\begin{proof}
It follows from Corollary \ref{cor_r_i=deg(disc)} that $r_i\neq 0$ iff $\pi:X\dashrightarrow \mathbb P^{i+1}$ has nonempty smooth discriminant, which by Theorem \ref{Thm2proof}, is equivalent to $\ddef(X)\le i$.
\end{proof}

\begin{theorem}\cite{Pie78}
    If $H$ is a general hyperplane, then $r_i(X)=r_{i-1}(X\cap H)$.
\end{theorem}
\begin{proof}
Let $H$ be the hyperplane corresponding to a general hyperplane $\mathbb P^i\subseteq \mathbb P^{i+1}$, the by restriction the generic projection $\pi: X\dashrightarrow \mathbb P^{i+1}$ to $H$, we have  $\pi|_H:X\cap H\dashrightarrow \mathbb P^{i}$, whose smooth discriminant locus $\Delta(\pi|_H)$ is precisely the intersection $\Delta(\pi)\cap \mathbb P^i$.
\end{proof}
By induction on taking hyperplane sections, one obtain
\begin{theorem}[Segre, 1912] \cite[p.190]{Kle86}\label{thm_Segre12} There is an equality
    $$r_i(X)=r_0(X_i),$$
where $X_i$ is a section of $X$ by a general codimension $i$ linear subspace.
\end{theorem}
We will see a simple application of Theorem \ref{thm_Segre12} in the next section.


\section{Monodromy and braid groups}\label{sec_braidgroup}

Throughout this section, we work with $\Bbbk=\C$. The goal of this section is to prove Theorem \ref{intro_thm_pi1}. We start with some general facts about discriminant locus of general finite projections of hypersurfaces onto projective spaces.

Let $\Sym^k (\mathbb{P}^1)$ be the symmetric product of $\mathbb P^1$. Let $D^k =\{[x_1,\dotsb, x_k]\in\Sym^k (\mathbb{P}^1)\vert\exists\ 1\leq i<j\leq k,\ x_i=x_j\}$ be the fat diagonal. Then the (spherical) \emph{braid group} $B_k$ of $k$ points on $\mathbb{P}^1$ is the fundamental group $$B_k\coloneqq\pi_1\left(\Sym^k (\mathbb{P}^1)\setminus D^k,\ast\right)$$ for some chosen base point $\ast\in \Sym^k (\mathbb{P}^1)\setminus D^k$.\par

Recall from Section \ref{sec_disc} that for $Y\subseteq\mathbb{P}^N$ and $p\colon Y\dashrightarrow \mathbb{P}^k$ any linear projection, the  (set-theoretical) discriminant locus of $p$ is $\Delta_{Fitt}$. We write $\Delta=\Delta_{Fitt}$ if it coincides with the smooth discriminant. In this section, we use $p$ as the notation for generic projection instead of $\pi$ to avoid confusion with the notation $\pi_1$ for the fundamental group.
\subsection{Braid monodromy from generic finite projection}
The following lemma can also be found in \cite{Moishezon81}.
\begin{lemma}\label{lmm:curve-braid}
    Assume that $Z$ is an irreducible plane curve with $p\colon Z\to \mathbb P^1$ a linear projection.
    Then there is a braid monodromy representation
$$g\colon \pi_1(\mathbb P^1\setminus \Delta_{Fitt},*)\to B_m$$
where $m=\deg(Z)$. Moreover, if $Z$ is smooth and the projection center is in general position, then the map $g$ is surjective.
\end{lemma}
\begin{proof}
    Write $F= p^{-1} (\ast)$. Note that the restriction $p^{\circ}:Z\setminus p^{-1}(\Delta_{Fitt})\to\mathbb{P}^1\setminus\Delta_{Fitt}$ is \'etale, so the fiber over
$t$ is an unordered configuration of 
$m$ distinct points in the line $\mathbb P^1$. This induces a continuous map $\mathbb{P}^1\setminus\Delta_{Fitt}\to\Sym^m (\mathbb{P}^1)\setminus D^m$, hence a group homomorphism $$\pi_1\left(\mathbb{P}^1\setminus\Delta_{Fitt} ,\ast\right)\to\pi_1\left(\Sym^m (\mathbb{P}^1)\setminus D^m,\ast\right)=B_m.$$ 
    
    When $Z$ is smooth and the projection is general, the projection $Z\to\mathbb{P}^1$ is simply branched, and $\Delta_{Fitt}=\Delta$ consists of branch points. The monodromy action of each generating loop of $\pi_1\left(\mathbb{P}^1\setminus\Delta ,\ast\right)$ is given by Dehn twist along a pair of points. As the covering space $p^{\circ}$ is connected, the monodromy action on $F$ is transitive, and there exists an ordering $F=\{z_0,\dotsb, z_{m-1}\}$ and generating loops $\gamma_0,\dotsb,\gamma_{m-1}$ of $\pi_1\left(\mathbb{P}^1\setminus\Delta ,\ast\right)$ such that $\gamma_i$ is the Dehn twist along the pair $(z_{i\mod (m-1)}, z_{i+1 \mod (m-1)})$. It is classically known that the action of the subgroup $\langle\gamma_0,\dotsb,\gamma_{m-1}\rangle$ on $p^{-1} (\ast)\setminus Z$ is exactly the braid relation \cite{Artin25}, and therefore the image of $\langle\gamma_0,\dotsb,\gamma_{m-1}\rangle$ under the map $\pi_1\left(\mathbb{P}^1\setminus\Delta ,\ast\right)\to B_m$ coincides with the braid group, so the map is surjective.
\end{proof}

\begin{remark}\normalfont
   Note that there is a surjection $B_m\to\Mod (\mathbb{P}^1\setminus F)$ to the mapping class group of $\mathbb{P}^1\setminus F$ (see \cite[Section 9.1]{FM12}). The composition of $g$ with this map is the induced monodromy action on $\mathbb{P}^1\setminus F$. 
\end{remark}

\begin{proposition}\label{prop:surjection-permutation-irreducible}
     Let $Y\subseteq\mathbb{P}^{N+1}$ be an irreducible hypersurface and let $p:Y\to \mathbb P^N$ be the finite branched cover induced from a linear projection $\mathbb{P}^{N+1}\dashrightarrow\mathbb{P}^N$. Fix a regular value $\ast\in\mathbb{P}^N\setminus \Delta_{Fitt}$ of $p$. Then 
     
     \begin{itemize}
         \item[(1)]  There is a braid monodromy representation \[
    \varphi: \pi_1\left(\mathbb{P}^N\setminus\Delta_{Fitt},\ast\right)\to B_m
     \] 
     where $m=\deg (Y)$.

     \item[(2)] Assume moreover that $Y$ is normal. Then $\Delta_{Fitt}=\Delta$, and
    \begin{itemize}
     
        \item[(a)] the induced braid monodromy representation
        \[
            \varphi:\pi_1(\mathbb P^N\setminus \Delta,\ast)\to B_m
        \]
        is surjective;

        \item[(b)] $\Delta$ is a hypersurface of degree
        \[
            \deg(\Delta)=m(m-1);
        \]

        \item[(c)] if $\mathcal X=Y^\perp\subseteq (\mathbb P^{N+1})^\vee$ and
        $X=H\cap \mathcal X$ is the general hyperplane section determined by the projection,
        then
        \[
            \Delta \cong X^\perp;
        \]

        \item[(d)] if $N>1$, then $\Delta$ is irreducible.
  
    \end{itemize}
     \end{itemize}
  
 \end{proposition}

\begin{proof}
   For part (1), note that away from the discriminant locus $\Delta_{Fitt}$ the projection $$p^{-1} (\mathbb{P}^N\setminus \Delta_{Fitt})\to\mathbb{P}^N\setminus \Delta_{Fitt}$$is \'etale. For the same reason as the proof of Lemma \ref{lmm:curve-braid}, we obtain a braid monodromy representation $\varphi:\pi_1\left(\mathbb{P}^N\setminus\Delta_{Fitt},\ast\right)\to B_m$. 
    
   For part (2), Suppose $Y$ is normal, 
  Then Lemma \ref{lemma_finite-clean} implies that $\Delta_{Fitt}=\Delta$ coincides with the smooth discriminant.
For a general line $\mathbb{L}\subseteq\mathbb{P}^N$, let $Z:=p^{-1} (\mathbb{L})$. Then $Z$ is a general plane section of $Y$. Since $Y$ is normal and smooth in codimension one, Bertini theorem implies that $Z$ is a smooth plane curve of degree $m$. The projection $\mathbb{P}^2\dashrightarrow\mathbb{L}$ restricts to a finite branched cover 
    $$Z=Y\cap\mathbb{P}^2\to \mathbb{L}$$
    with branching locus $\Delta\cap\mathbb{L}$. As $Z$ is smooth and the choice of $\mathbb{L}$ is general, this is simply branched. Then, applying Lemma \ref{lmm:curve-braid}, we get a surjective group homomorphism $\pi_1\left(\mathbb{L}\setminus\Delta,\ast\right)\twoheadrightarrow B_m$. Finally, note that there is a commutative diagram of groups
    \[
    \begin{tikzcd}
        \pi_1\left(\mathbb{L}\setminus\Delta,\ast\right)\arrow[r]\arrow[rr,bend right=20]&\pi_1\left(\mathbb{P}^N\setminus\Delta,\ast\right)\arrow[r]&B_m,
    \end{tikzcd}
    \]and hence surjectivity of $\pi_1\left(\mathbb{L}\setminus\Delta,\ast\right)\twoheadrightarrow B_m$ implies the surjectivity of the second map. This proves (a)\par

   For (b), we note $\deg(\Delta)$ is the number of the (simple) branch points of $Z\to \mathbb P^1$. Specifically, it follows from Theorem \ref{thm_Segre12} that
 $$\deg(\Delta)=\deg(Z^{\vee})=m(m-1).$$
 
 The statement (c) follows from Theorem \ref{thm_finite-normal}. 
For irreducibility of $\Delta$, by Theorem \ref{thm_finite-normal} it suffices to show
that $\dim Y^\perp>1$. Suppose instead that $Y^\perp$ is a curve.
Let $S$ be a general $\mathbb P^3$-section of $Y$. Since $Y$ is normal,
$S$ is a normal surface. By Theorem \ref{thm_mainthmequiv}, the surface $S$
is the dual variety of the image of a general projection of $Y^\perp$ to $(\mathbb P^3)^\vee$.
But the dual variety of a nondegenerate space curve is a tangential surface,
hence is singular along the osculating curve
(\cite[Proposition 10.24]{Dolgachev-classicalAG}), contradicting the normality of $S$.
\end{proof}

\begin{remark}\normalfont
Note that by Zariski's theorem on fundamental group of complement of hypersurface \cite[Theorem 3.22]{Voisin}, the $\pi_1$ above is isomorphic to fundamental group of complement of plane curve 
$\pi_1(\mathbb P^2\setminus C,*)$, where $C=X^{\perp}\cap \mathbb P^2$ is a general plane section of degree $m(m-1)$. In particular, when $m\ge 4$,  Proposition \ref{prop:surjection-permutation-irreducible} (2a) implies that $\pi_1(\mathbb P^2\setminus C,*)$ contains a free group with two generators, and is "big" in the sense of \cite{large-pi1}. Note that the plane curves $C$ here are in general different from the examples studied in \cite{large-pi1}.
\end{remark}
\begin{remark}\normalfont
    There is a natural map from braid group $B_m \to \mathfrak S_m$ to the symmetric group, and the composition map 
$$\rho:\pi_1\left(\mathbb{P}^N\setminus\Delta_{Fitt},\ast\right)\to B_m \to \mathfrak S_m$$
is the monodromy representation permuting the $m$ points. 

In the set-up of Proposition \ref{prop:surjection-permutation-irreducible} (1) and without additional assumption, one can still show $\rho$ is surjective. This is an application of the Uniform Principle Theorem of an irreducible curve \cite[p.111]{ACGH85}.
\end{remark}

Now we prove Theorem \ref{intro_thm_pi1} from the introduction.

\begin{proof}[Proof of Theorem \ref{intro_thm_pi1}] It follows from Proposition \ref{prop:surjection-permutation-irreducible} (2). 
\end{proof}
\subsection{Varieties with normal duals}
 The following is a direct consequence of Proposition \ref{prop:surjection-permutation-irreducible} by assuming $\mathcal X$ is smooth. This proves Corollary \ref{intro_cor_normaldual}.
\begin{proposition}[Corollary \ref{intro_cor_normaldual}]\label{Prop_normaldual}
    Let $\mathcal X\subseteq\mathbb P^{N+1}$ be a smooth projective variety. Suppose that its dual variety $\mathcal X^{\perp}$ is a normal hypersurface of degree $m$. Then the dual variety $X^{\perp}$ of its general hyperplane section $X=\mathcal X\cap H$ has degree $m(m-1)$. Moreover, the braid monodromy representation induces a surjection 
$$\pi_1((\mathbb P^N)^{\vee}\setminus X^{\perp},*)\twoheadrightarrow B_m.$$
\end{proposition}

\begin{example}\normalfont
   Let $\mathcal X=v_2(\mathbb P^n)$ be the second Veronese embedding of $\mathbb P^n$. Then
$\mathcal X^\perp$ is the determinantal hypersurface of singular symmetric
$(n+1)\times (n+1)$ matrices, so $\deg(\mathcal X^\perp)=n+1$.
For a general hyperplane section $X=\mathcal X\cap H$, the dual variety
$X^\perp$ is a hypersurface of degree $n(n+1)$, and the braid monodromy induces
a surjection
\[
\pi_1((\mathbb P^N)^\vee\setminus X^\perp,*)\twoheadrightarrow B_{n+1}.
\]
\end{example}
\begin{example}\normalfont\label{example_B4} 
In the special case $n=2$, $\mathcal X=v_2(\mathbb P^2)$ and $\mathcal X^\perp$
is a normal cubic hypersurface. (See Section \ref{sec_Veronese-surface}.) A general hyperplane section $X=\mathcal X\cap H$
is the rational normal quartic, and $X^\perp$ is the discriminant hypersurface in
$\Sym^4(\mathbb P^1)\cong \mathbb P^4$. A general plane section of $X^\perp$ is the
classical rational sextic with $4$ nodes and $6$ cusps studied by Zariski \cite{Zar36}.
Its complement has fundamental group $B_4$. Hence in this case the braid monodromy
surjection
$\pi_1((\mathbb P^4)^\vee\setminus X^\perp,*)\twoheadrightarrow B_3$
is not an isomorphism.
\end{example}


\bibliographystyle{alpha}
\bibliography{bibfile}

\end{document}